\documentclass[11pt]{amsart}
\usepackage{mathrsfs}
\usepackage{geometry}
\usepackage{graphicx,color}

\theoremstyle{plain}
\newtheorem{theorem}{Theorem} [section]

\newtheorem{lemma}[theorem]{Lemma}
\newtheorem{prop}[theorem]{Proposition}
\newtheorem*{teoA}{Main Theorem}

\newtheorem*{kj}{Kohler-Jobin inequality}
\newtheorem*{conj}{Conjecture}
\theoremstyle{definition}
\newtheorem{definition}[theorem]{Definition}
\newtheorem{oss}[theorem]{Remark}
\newtheorem*{ack}{Acknowledgements}

\numberwithin{theorem}{section}
\numberwithin{equation}{section}
\numberwithin{figure}{section}

\def\mean#1{\mathchoice
         {\mathop{\kern 0.2em\vrule width 0.6em height 0.69678ex depth -0.58065ex
                 \kern -0.8em \intop}\nolimits_{\kern -0.4em#1}}%
         {\mathop{\kern 0.1em\vrule width 0.5em height 0.69678ex depth -0.60387ex
                 \kern -0.6em \intop}\nolimits_{#1}}%
         {\mathop{\kern 0.1em\vrule width 0.5em height 0.69678ex
             depth -0.60387ex
                 \kern -0.6em \intop}\nolimits_{#1}}%
         {\mathop{\kern 0.1em\vrule width 0.5em height 0.69678ex depth -0.60387ex
                 \kern -0.6em \intop}\nolimits_{#1}}}

\def\R{\mathbb R}

\def\eps{\varepsilon}

\def\A{\mathcal A}
\def\F{\mathcal F}
\def\G{\mathcal G}

\DeclareMathOperator*{\osc}{osc}
\DeclareMathOperator{\dist}{dist}

\title{Faber-Krahn inequalities in sharp quantitative form}

\author[Brasco]{Lorenzo Brasco}
\address{{\bf L. B.} Laboratoire d'Analyse, Topologie, Probabilit\'es, Aix-Marseille Universit\'e, 39 Rue Fr\'ed\'eric Joliot Curie, 13453 Marseille Cedex 13, France}
\email{lorenzo.brasco@univ-amu.fr}

\author[De Philippis]{Guido De Philippis}
\address{{\bf G. D. P.} Hausdorff Center for Mathematics, Endenicher Allee 62, D-53115 Bonn, Germany}
\email{guido.de.philippis@hcm.uni-bonn.de}

\author[Velichkov]{Bozhidar Velichkov}
\address{{\bf B. V.} Scuola Normale Superiore di Pisa Piazza dei Cavalieri 7, 56126 Pisa, Italy}
\email{b.velichkov@sns.it}

\begin{document}

\begin{abstract}
The classical Faber-Krahn inequality asserts that balls (uniquely) minimize the first eigenvalue of the Dirichlet-Laplacian among sets with given volume. In this paper we prove a sharp quantitative enhancement of this result, thus confirming a conjecture by Nadirashvili and Bhattacharya-Weitsman. More generally, the result applies to every optimal Poincar\'e-Sobolev constant for the embeddings $W^{1,2}_0(\Omega)\hookrightarrow L^q(\Omega)$.
\end{abstract}
\maketitle

\section{Introduction}\label{intro}

\subsection{Background}
Let $\Omega\subset\mathbb{R}^N$ be an open set with finite measure, we denote by $W^{1,2}_0(\Omega)$ the closure of $C^\infty_0(\Omega)$ in the norm
\[
\|u\|_{W^{1,2}_0(\Omega)}=\left(\int_{\Omega} |\nabla u|^2\, dx\right)^{1/2}.
\]
The {\it first eigenvalue of the Dirichlet-Laplacian} of $\Omega$ is defined by
\[
\lambda(\Omega)=\min_{u\in W^{1,2}_0(\Omega)}\left\{\int_\Omega |\nabla u|^2\, dx\, :\, \|u\|_{L^2(\Omega)}=1\right\}.
\]
The quantity $\lambda(\Omega)$ is also called {\it principal frequency} of the set $\Omega$. If we denote by $\Delta$ the usual Laplace operator, $\lambda(\Omega)$ coincides with the smallest real number $\lambda$ for which the Helmholtz equation
\[
-\Delta u=\lambda\, u\quad \mbox{ in }\Omega,\qquad u=0,\quad \mbox{ on }\partial\Omega,
\]
admits non-trivial solutions. 
\par
A classical optimization problem connected with $\lambda$ is the following one: among sets with given volume, find the one which minimizes the principal frequency $\lambda$. Actually, balls are the (only) solutions to this problem. 
As $\lambda$ has the dimensions of a length to the power $-2$, this ``isoperimetric'' property can be equivalently rewritten as 
\begin{equation}
\label{faber}
|\Omega|^{2/N}\,\lambda(\Omega)\ge |B|^{2/N}\,\lambda(B),
\end{equation}
where $B$ denotes a generic $N-$dimensional ball and $|\cdot|$ stands for the $N-$dimensional Lebesgue measure of a set. Moreover, equality holds in \eqref{faber} if and only if $\Omega$ is a ball. The estimate \eqref{faber} is the celebrated {\it Faber-Krahn inequality}. We recall that the usual proof of  this inequality relies on the so-called {\it Schwarz symmetrization} (see \cite[Chapter 2]{He}). The latter consists in associating to each positive function $u\in W^{1,2}_0(\Omega)$ a radially symmetric decreasing function $u^*\in W^{1,2}_0(B_\Omega)$, where $B_\Omega$ is the ball centered at the origin such that $|B_\Omega|=|\Omega|$. The function $u^*$ is {\it equimeasurable} with $u$, that is
\[
|\{x\, :\, u(x)>t\}|=|\{x\, :\, u^*(x)>t\}|,\qquad \mbox{ for every } t\ge 0,
\]
so that in particular every $L^q$ norm of the function $u$ is preserved.
More interestingly, one has the well-known {\it P\'olya-Szeg\H{o} principle}
\begin{equation}
\label{PS}
\int_{B_\Omega} |\nabla u^*|^2\, dx\le \int_\Omega |\nabla u|^2\, dx,
\end{equation}
from which the Faber-Krahn inequality easily follows.
\vskip.2cm
The fact that balls can be characterized as the only sets for which equality holds in \eqref{faber}, naturally leads to consider the question of its {\it stability}. More precisely, one would like to improve \eqref{faber}, by adding in its right-hand side a reminder term measuring the deviation of a set $\Omega$ from spherical symmetry. A typical quantitative Faber-Krahn inequality then reads as follows
\begin{equation}
\label{enforced}
|\Omega|^{2/N}\,\lambda(\Omega)- |B|^{2/N}\,\lambda(B)\ge g(d(\Omega)),
\end{equation}
where $g$ is a modulus of continuity and $\Omega\mapsto d(\Omega)$ is some scaling invariant {\it asymmetry functional}.
The quest for quantitative versions like \eqref{enforced} is not new and has attracted an increasing interest in the last years. To the best of our knowledge, the first ones to prove partial results in this direction have been Hansen and Nadirashvili in \cite{HN} and Melas in \cite{Me}. Both papers treat the case of simply connected sets in dimension $N=2$ or the case of convex sets in general dimensions. These pioneering results prove inequalities like \eqref{enforced}, with a modulus of continuity (typically a power function) depending on the dimension $N$ and with the following asymmetry functionals\footnote{In the paper \cite{Me}, the quantitative result is stated in a slighlty different form, but it is not difficult to see that it can be written as in \eqref{enforced}, by using the functional $d_2(\Omega)$.}
\[
d_1(\Omega)=1-\frac{r_\Omega}{r_{B_\Omega}},\qquad\qquad \mbox{ where }\quad \begin{array}{rl} 
r_\Omega=&\mbox{\,inradius of }\Omega,\\
 r_{B_\Omega}=&\mbox{radius of the ball $B_\Omega$},
\end{array}
\]
like in \cite{HN}, or
\[
d_2(\Omega)=\min\left\{\frac{\max\{|\Omega\setminus B_1|,|B_2\setminus \Omega|\}}{|\Omega|}\, :\, B_1\subset \Omega\subset B_2 \mbox{ balls}\right\},
\]
as in \cite{Me}.
It is easy to see that for general sets an estimate like \eqref{enforced} with the previous asymmetry functionals {\it can not be true} (just think of a ball with a small hole at the center). In the general case, a better notion of asymmetry is the so called {\it Fraenkel asymmetry}, defined as 
\begin{equation}
\label{eq:asymmetry}
\mathcal{A}(\Omega)=\inf\left\{\frac{|\Omega\Delta B|}{|B|} \, :\,  \text{ \(B\) ball such that \(|B|=|\Omega|\)}\right\},
\end{equation}
where the symbol $\Delta$ now stands for the symmetric difference between sets. For such an asymmetry functional, Bhattacharya and Weitsman \cite{BW} and Nadirashvili \cite{Na} indipendently conjectured the following.  
\begin{conj}
There exists a dimensional constant $\sigma>0$ such that
\begin{equation}
\label{BWN}
|\Omega|^{2/N}\,\lambda(\Omega)- |B|^{2/N}\,\lambda(B)\ge \sigma\, \mathcal{A}(\Omega)^2.
\end{equation}
\end{conj}
\noindent
In this paper we provide a positive answer to the above conjecture.
\vskip.2cm
Let us notice that the previous result is {\it sharp}, since the power $2$ on the asymmetry can not be replaced by any smaller power.
Indeed one can verify that the family of ellipsoids
\[
\Omega_{\varepsilon}=\Big\{(x',x_N)\in \mathbb{R}^{N}\, :\, |x'|^2+(1+\varepsilon)\,x_N^2\le 1\Big\},\qquad 0<\varepsilon\ll 1,
\]
are such that
\[
\mathcal{A}(\Omega_\varepsilon)\simeq \varepsilon\qquad \mbox{ and }\qquad |\Omega_\varepsilon|^{2/N}\,\lambda(\Omega_\varepsilon)- |B|^{2/N}\,\lambda(B)\simeq \varepsilon^2 .
\]
We mention that the following weaker version of \eqref{BWN} was already known, 
\[
|\Omega|^{2/N}\,\lambda(\Omega)-|B|^{2/N}\,\lambda(B)\ge\sigma\,\begin{cases}
\mathcal{A}(\Omega)^3,& N=2,\\
\mathcal{A}(\Omega)^4,& N\ge 3,
\end{cases}
\]
obtained by Bhattacharya \cite{Bha} (for the case $N=2$) and more recently by Fusco, Maggi and Pratelli in \cite{FMP1} for the general case. For ease of completeness, we also mention 
\cite{Po} and \cite{Sni} for similar partial results and some probabilistic applications.

\subsection{The result of this paper}
Actually, we are going to prove a slightly more general version of \eqref{BWN}.
To state our result, let us consider the following optimal Poincar\'e-Sobolev constants for the embedding \(W^{1/2}_0(\Omega)\hookrightarrow L^{q}(\Omega)\)
\begin{equation}
\label{autolavoro}
\lambda_{2,q}(\Omega)=\min_{u\in W^{1,2}_0(\Omega)}\left\{\int_\Omega |\nabla u|^2\, dx\, :\, \|u\|_{L^q(\Omega)}=1\right\},
\end{equation}
where the exponent $q$ satisfies
\begin{equation}
\label{q}
1\le q<2^*:=\left\{\begin{array}{lr}
\displaystyle \frac{2\,N}{N-2},&\qquad \mbox{ if } N\ge 3,\\
&\\
+\infty,& \mbox{ if }N=2.\\
\end{array}
\right.
\end{equation}
Of course, when $q=2$ we are back to the principal frequency mentioned at the beginning.
We also point out that for $q=1$, the quantity $1/\lambda_{2,1}(\Omega)$ is usually referred to as the {\it torsional rigidity} of the set $\Omega$. 
Observe that the shape functional $\Omega\mapsto \lambda_{2,q}(\Omega)$ verifies the scaling law
\[
\lambda_{2,q}(t\,\Omega)=t^{N-2-\frac{2}{q}\, N}\, \lambda_{2,q}(\Omega),
\]
the exponent $N-2-2/q\, N$ being negative. In particular, the quantity
\[
|\Omega|^{\frac{2}{N}+\frac{2}{q}-1}\, \lambda_{2,q}(\Omega),
\]
is scaling invariant. 
Still by means of Schwarz symmetrization, the following general family of Faber-Krahn inequalities can be derived 
\begin{equation}
\label{FK}
|\Omega|^{\frac{2}{N}+\frac{2}{q}-1}\, \lambda_{2,q}(\Omega)\ge |B|^{\frac{2}{N}+\frac{2}{q}-1}\, \lambda_{2,q}(B),
\end{equation}
where $B$ is any $N-$dimensional ball. Again, equality in \eqref{FK} is possible if and only if $\Omega$ is a ball. 
The main result of the paper is the following sharp quantitative improvement of \eqref{FK}.
\begin{teoA}
\label{thm:fkstability}
Let $q$ be an exponent verifying \eqref{q}. There exists a constant  \(\sigma_{2,q}\), depending only on the dimension \(N\) and $q$, such that for every open set \(\Omega\subset \mathbb{R}^N\) with finite measure we have
\begin{equation}
\label{eq:fkstability}
|\Omega|^{\frac{2}{N}+\frac{2}{q}-1}\,\lambda_{2,q}(\Omega)-|B|^{\frac{2}{N}+\frac{2}{q}-1}\, \lambda_{2,q}(B)\ge \sigma_{2,q}\, \mathcal{A}(\Omega)^2.
\end{equation}
\end{teoA}
As already mentioned, by choosing \(q=2\) we obtain a  proof of the Bhattacharya-Weitsman and Nadirashvili  Conjecture. 
\par
We also remark that, as explained in \cite[Remark 3.6]{BP}, the above Theorem allows to improve the exponent in the quantitative stability inequality for the second Dirichlet eigenvalue of the Laplacian proved in \cite[Theorem 3.5]{BP}. 

\subsection{Strategy of the proof}
We start recalling the usual strategy used to derive quantitative versions of Faber-Krahn inequalities. As the proof of \eqref{FK} is based on the P\'olya-Szeg\H{o} principle \eqref{PS},  the central core of all the already exhisting stability results is represented by 
Talenti's proof of \eqref{PS} (see \cite{Ta1}). This combines the Coarea Formula, the convexity of the function $t\mapsto t^2$ and 
the standard Isoperimetric Inequality 
\begin{equation}
\label{isosciarpa}
|\Omega|^\frac{1-N}{N}\,P(\Omega)\ge |B|^\frac{1-N}{N}\, P(B),
\end{equation}
applied to the superlevel sets of 
a function $u$ achieving $\lambda_{2,q}(\Omega)$, where $P(\cdot)$ denotes the perimeter of a set. The main idea of the papers \cite{Bha,FMP1,HN} and \cite{Me} is that
 of substituting the classical isoperimetric statement \eqref{isosciarpa} with an improved quantitative 
version. For simply connected sets in dimension $N=2$ or for convex sets one can appeal to the so called {\it Bonnesen inequalities} 
(see \cite{Os}), like in \cite{Bha,HN,Me}. More generally, one can apply the striking recent result of \cite{FMPiso}, proving a sharp quantitative version of \eqref{isosciarpa} valid for every set and every dimension. Then the main difficulty is that of estimating the ``propagation of asymmetry'' 
from the superlevel sets of the optimal function $u$ to the whole domain $\Omega$. This is a very delicate step, which usually results in 
a (non sharp) estimate like the ones recalled above. It is worth mentioning the recent paper \cite{BCFP} for some recent developments on quantitative versions of the P\'olya-Szeg\H{o} principle.
\vskip.2cm
In this paper on the contrary, we use a different strategy. Indeed, the proof of our Main Theorem is based on the \emph{selection principle} introduced by Cicalese and Leonardi in \cite{CL} to give a new proof of the previously recalled quantitative isoperimetric inequality of \cite{FMPiso}.
\par
The  selection principle turns out to be a very flexible technique and after the paper \cite{CL} it has been applied to a wide variety of geometric problems, see for instance \cite{AFM, BDF} and \cite{DM}. Up to now however it has been used only for problems where the main term is given, roughly speaking, by the \empty{perimeter} of \(\Omega\). As we will explain below, this is due to the fact the selection principle highly relies on the regularity theory for sets  minimizing some (perturbed) shape functional. If the dominating term of the functional is given by a area-type term, then well developed techniques in Geometric Measure Theory ensure the desired regularity.
\vskip.2cm
Let us now  explain the main ideas behind our proof. First by an application of the \emph{Kohler-Jobin inequality} (\cite{KJ}) we will show in Section \ref{sec:KJ} that \eqref{eq:fkstability} is implied by the following  inequality
\begin{equation}
\label{energyintro}
E(\Omega)-E(B_1)\ge \sigma\,\mathcal A(\Omega)^2,\qquad\mbox{ for every $\Omega$ such that } |\Omega|=|B_1|,
\end{equation}
where \(\sigma\) is a  dimensional constant and $B_1$ is the ball of radius $1$ and centered at the origin. Here \(E(\Omega)\) is the \emph{energy functional} of \(\Omega\),
\begin{equation}
\label{eq:torsionintro}
E(\Omega)=\min_{u\in W^{1,2}_0(\Omega)} \frac{1}{2}\, \int_\Omega |\nabla u|^2\, dx-\int_\Omega u\, dx= \frac{1}{2}\, \int_\Omega |\nabla u_\Omega|^2\, dx-\int_\Omega u_{\Omega}\, dx,
\end{equation}
where \(u_\Omega\in W^{1,2}_{0}(\Omega)\) is the (unique) function achieving the above minimum.
\par
Suppose now by contradiction that \eqref{energyintro} is false. Since it is pretty easy to see that \eqref{energyintro} can only fail in the small asymmetry regime (i.e. on sets converging in $L^1$ to the ball), we find a sequence of sets \(\Omega_j\) such that
\begin{equation}
\label{contraintro}
|\Omega_j|=|B_1|, \qquad \varepsilon_j:=\mathcal A (\Omega_j)\to 0\qquad \text{and} \qquad E(\Omega_j)-E(B_1)\le \sigma \mathcal A(\Omega_j)^2,
\end{equation}
with \(\sigma\) as small as we wish. We now look for an ``improved''  sequence of sets \(U_j\), still contradicting \eqref{energyintro} and enjoying some additional smoothness properties. In the spirit of Ekeland's variation principle, these sets will be selected through some minimization problem. Roughly speaking we look for sets \(U_j\) which solve the following 
\begin{equation}
\label{intromin}
\min\Big\{E(\Omega)+\sqrt{\varepsilon_j^2+\sigma(\mathcal A(\Omega)-\varepsilon_j)^2}\,:\, |\Omega|=|B_1| \Big\}.
\end{equation}
One can easily show that the sequence \(U_j\) still contradict \eqref{energyintro} and  that \(\mathcal A(U_j)\to 0\) (see Lemma \ref{lm:prop1}). Relying on the minimality  of \(U_j\), one then would like to show that the \(L^1\) convergence to \(B_1\) can be improved to a smooth convergence. If this is the case, then the second order expansion of \(E(\Omega)\) for smooth  nearly spherical sets  done in Section \ref{sec:fuglede}  shows that \eqref{contraintro} cannot hold true if \(\sigma\) is sufficiently small. 
\par
The key point is thus to prove (uniform) regularity estimates for sets solving \eqref{intromin}. For this, first one would like to get rid of volume constraints applying some sort of Lagrange multiplier principle to show that \(U_j\) minimizes
\begin{equation}
\label{lagrange}
E(\Omega)+\sqrt{\varepsilon_j^2+\sigma(\mathcal A(\Omega)-\varepsilon_j)^2}+\Lambda\,|\Omega|.
\end{equation}
Then, taking advantage of the fact that we are considering  a ``min--min'' problem, the previous is equivalent to require that $u_j=u_{U_j}$ minimizes  
\begin{equation}
\label{funzioni}
\begin{split}
\frac 1 2 \int_{\mathbb{R}^N} |\nabla v|^2\,dx -\int_{\mathbb{R}^N} v\, dx +\Lambda\,\big |\{v>0\}\big|+\sqrt{\varepsilon_j^2+\sigma(\mathcal A(\{v>0\})-\varepsilon_j)^2},
\end{split}
\end{equation}
among all functions with compact support.
Since we are now facing a perturbed free boundary type problem, we aim to apply the techniques of Alt and Caffarelli \cite{AC} (see also \cite{B,BL}) to show the regularity of \(\partial U_j=\partial\{u_{j}>0\}\) and  to obtain  the smooth convergence of \(U_j\) to \(B_1\).
\vskip 0.2cm
Even if this will be the general strategy, several non-trivial modifications have to be done to the above sketched proof. First of all, although solutions to \eqref{funzioni} enjoy some mild regularity property, we cannot expect  \(\partial \{u_j>0\}\) to be smooth.  Indeed, by formally computing the optimality condition\footnote{That is differentiating the functional along  perturbation of the form \(v_t=u_j\circ ({\rm Id}+tV)\) where \(V\) is a smooth vector field, see Appendix \ref{sec:shape} and Lemma \ref{lm:EL} below.} of \eqref{funzioni} and assuming that \(B_1\) is the unique optimal ball for \(\{u_j>0\}\) in \eqref{eq:asymmetry}, one gets that \(u_j\) should satisfy
\[
\left|\frac{\partial u_j}{\partial \nu}\right|^2=\Lambda+\frac{\sigma(\mathcal A(\{u_j>0\})-\varepsilon_j)}{\sqrt{\varepsilon_j^2+\sigma(\mathcal A(\{u_j>0\})-\varepsilon_j)^2}}\,\big(1_{\mathbb{R}^N\setminus \overline B_1}-1_{B_1}\big),\qquad \text{on \(\partial \{u_j>0\},\)}
\]
where $1_A$ denotes the characteristic function of a set $A$ and $\nu$ is the outer normal versor.
This means that  the normal derivative of \(u_j\) is discontinuous at points where \(U_j=\{u_j>0\}\) crosses \(\partial B_1\). Since 
classical elliptic regularity implies that if \(\partial U_j\) is \(C^{1,\gamma}\) then \(u_j\in C^{1,\gamma}(\overline{U_j})\), it is clear that the sets \(U_j\) can not enjoy too much smoothness properties.
\par
To overcome this difficulty, inspired by \cite{ATW}, we replace the Fraenkel asymmetry with a new ``distance'' between a set \(\Omega\) and the set of balls, which behaves like a squared \(L^2\) distance between the boundaries (see Definition \ref{defalpha}). In particular it dominates the square of the Fraenkel asymmetry (see Lemma \ref{lem:alphaprop}) and it is differentiable with respect to the variations needed to compute the optimality conditions (see Lemma \ref{lm:EL}).
\par
A second technical difficulty is that no global  Lagrange multiplier principle is available. Indeed, since the energy $E$ is negative and
\[
E(t\, \Omega)=t^{-N-2}\, E(\Omega)\qquad \mbox{ and }\qquad |t\, \Omega|=t^N\, |\Omega|,\qquad t>0,
\]
by a simple scaling argument one sees that the infimum of \eqref{lagrange} is identically  \(-\infty\). Reducing to a priori bounded set and following \cite{AAC}, we can however 
replace the term \(\Lambda\, |\Omega|\) with a term of the form \(f(|\Omega|)\), for a suitable strictly increasing function vanishing when \(|\Omega|=|B_1|\), 
see Lemma \ref{torsionelimitato} below. At this point we are able to perform the strategy described above to obtain \eqref{energyintro} for uniformly bounded sets \(\Omega\), with a constant \( \sigma\) depending on \({\rm diam}(\Omega)\). 
\par
In Section \ref{sec:second} we will finally show how to pass from bounded to general sets. For this last step, 
we will take advantage of the non-optimal quantitative stability inequality proved in \cite{FMP1}.

\section{First step: reduction to the energy functional}
\label{sec:KJ}

For every $\Omega\subset\mathbb{R}^N$ open set with finite measure,  the {\it energy functional}  is defined as 
\begin{equation}
\label{eq:torsion}
E(\Omega)=\min_{u\in W^{1,2}_0(\Omega)} \frac{1}{2}\, \int_\Omega |\nabla u|^2\, dx-\int_\Omega u\, dx.
\end{equation}
The function $u_\Omega$ achieving the above minimum is unique and will be called {\it energy function of $\Omega$}, and it satisfies
\begin{equation}
\label{equationeenergia}
-\Delta u_\Omega =1\quad \mbox{ in }\Omega,\qquad u_\Omega =0,\quad \mbox{ on }\partial\Omega,
\end{equation}
in weak sense.
Multiplying the above equation by \(u_\Omega\) and integrating by parts one sees that
\begin{equation}\label{energyequivalent}
E(\Omega)=-\frac 1 2 \int_\Omega |\nabla u_{\Omega}|^2\,dx=-\frac1 2 \int_\Omega u_{\Omega}\,dx.
\end{equation}
By means of an easy homogeneity argument, we have
\begin{equation}
\label{enrgytorsion}
E(\Omega)=-\frac{1}{2}\, \max_{u\in W^{1,2}_0(\Omega)}\left\{\left(\int_\Omega u\, dx\right)^2\, :\, \|\nabla u\|_{L^2(\Omega)}=1\right\}=-\frac{1}{2}\, \frac{1}{\lambda_{2,1}(\Omega)}.
\end{equation}
In other words $E(\Omega)$ coincides with the opposite of the torsional rigidity of $\Omega$ (up to the multiplicative factor $1/2$). In particular we should pay attention to the fact that $E(\Omega)$ {\it is always a negative quantity}.
Then the Faber-Krahn inequality \eqref{FK} for $q=1$ can now be rewritten 
\begin{equation}
\label{sv}
E(\Omega)\, |\Omega|^{-\frac{N+2}{N}}\ge E(B)\, |B|^{-\frac{N+2}{N}},
\end{equation}
where $B$ is any ball and equality can hold if and only if $\Omega$ itself is a ball. Sometimes we will refer to this inequality as the {\it Saint-Venant inequality}.
\vskip.2cm
The quantity $\lambda_{2,q}$ defined in \eqref{autolavoro} and the energy functional are linked by the following ``isoperimetric'' inequality, due to Marie-Th\'er\`ese Kohler-Jobin (\cite[Theorem 3]{KJ82} and \cite[Th\'eor\`eme 1]{KJ}), see also \cite{Br} for some recent generalizations of this inequality. 
\begin{kj}
Let $q>1$ be an exponent verifying \eqref{q}. For every $\Omega\subset\mathbb{R}^N$ open set with finite measure, we have
\begin{equation}
\label{kj}
\lambda_{2,q}(\Omega)\, (-E(\Omega))^\vartheta\ge \lambda_{2,q}(B)\,(- E(B))^\vartheta,\qquad \mbox{ with }\ \vartheta(q,N)=\left(\displaystyle\frac{1}{q}-\frac{N-2}{2N}\right)\, \frac{2\,N}{N+2},
\end{equation}
where $B$ is any ball. Equality holds in \eqref{kj} if and only if $\Omega$ itself is a ball. 
\end{kj}
The next result shows that quantitative estimates for the energy functional $E$, automatically translate into estimates for the Faber-Krahn inequality.
\begin{prop}
\label{prop:trick}
Let $q>1$ be an exponent verifying \eqref{q}. Suppose that there exists a constant $\sigma_E>0$ such that 
\begin{equation}
\label{goal}
E(\Omega)\, |\Omega|^{-\frac{N+2}{N}}- E(B)\, |B|^{-\frac{N+2}{N}}\ge \sigma_E\, \mathcal{A}(\Omega)^2,
\end{equation}
for every open set $\Omega\subset\mathbb{R}^N$ with finite measure. Then we also have
\[
|\Omega|^{\frac{2}{N}+\frac{2}{q}-1}\,\lambda_{2,q}(\Omega)-|B|^{\frac{2}{N}+\frac{2}{q}-1}\, \lambda_{2,q}(B)\ge \sigma_{2,q}\, \A(\Omega)^2,
\]
for some constant $\sigma_{2,q}>0$ depending only on $\sigma_E$ and $q$.
\end{prop}
\begin{proof}
Without loss of generality, let us suppose that $|\Omega|=1$ and let $B$ be a ball having unit measure. By \eqref{kj} 
one obtains
\begin{equation}
\label{cappio}
\frac{\lambda_{2,q}(\Omega)}{\lambda_{2,q}(B)}-1\ge \left(\frac{E(B)}{E(\Omega)}\right)^\vartheta-1.
\end{equation}
By concavity, for every $0<\vartheta\le 1$ we have
\[
t^\vartheta-1\ge (2^{\vartheta}-1)\, (t-1),\qquad t\in[1,2].
\]
From \eqref{cappio} we can easily infer that if $-E(B)\le -2\, E(\Omega)$, then
\[
\frac{\lambda_{2,q}(\Omega)}{\lambda_{2,q}(B)}-1\ge c_\vartheta\, \left(\frac{E(B)}{E(\Omega)}-1\right)\ge \frac{c_\vartheta\,\sigma_E}{-E(B)}\, \mathcal{A}(\Omega)^2,
\]
where in the last inequality we used that $-E(\Omega)\le -E(B)$ by \eqref{sv}. On the other hand, if $-E(B)>-2\, E(\Omega)$, still by \eqref{cappio}
\[
\frac{\lambda_{2,q}(\Omega)}{\lambda_{2,q}(B)}-1\ge 2^\vartheta-1\ge \frac{2^\vartheta-1}{4}\, \mathcal{A}(\Omega)^2,
\]
since $\mathcal{A}(\Omega)<2$.
\end{proof}
\begin{oss}
It is well-known that for $N\ge 3$ we have
\[
\lim_{q\to 2^*} \lambda_{2,q}(\Omega)=\inf\left\{\int_{\mathbb{R}^N} |\nabla u|^2\, dx\, :\ u\in W^{1,2}_0(\Omega),\ \|u\|_{L^{2^*}(\mathbb{R}^N)}=1\right\},
\]
and the latter is the best costant in the Sobolev inequality, a quantity which does not depend on the set $\Omega$. Clearly this implies that the constant $\sigma_{2,q}$ in \eqref{eq:fkstability} must converge to $0$ as $q$ goes to $2^*$. A closer inspection of the proof of Proposition \ref{prop:trick} shows that 
\[
\sigma_{2,q}\simeq 2^{\vartheta(q,N)}-1\simeq (2^*-q),
\]
as $q$ goes to $2^*$. 
The conformal case $N=2$ is a little bit different. In this case we have (see \cite[Lemma 2.2]{RW})
\[
\lim_{q\to+\infty} \lambda_{2,q}(\Omega)=0\qquad \mbox{ and }\qquad \lim_{q\to+\infty} q\, \lambda_{2,q}(\Omega)=8\, \pi\, e,
\] 
for every set $\Omega$. The asymptotic behaviour of the constant $\sigma_{2,q}$ is then given by
\[
\sigma_{2,q}\simeq (2^{\vartheta(q,2)}-1)\, \lambda_{2,q}(B)\simeq \frac{8\, \pi\,e}{q^2},
\]
as $q$ goes to $+\infty$.
\end{oss}

\section{Second step: sharp stability for nearly spherical sets}
\label{sec:fuglede}

In this section we show the validity of a stronger form of \eqref{energyintro} for sets smoothly close to the ball $B_1$ of unit radius and centered at the origin. We start with two definitions.
\begin{definition}\label{almostspherical}
An open bounded set $\Omega\subset\mathbb{R}^N$ is said {\it nearly spherical of class $C^{2,\gamma}$ parametrized by \(\varphi\)}, if there exists $\varphi\in C^{2,\gamma}(\partial B_1)$ with $\|\varphi\|_{L^\infty}\le 1/2$, such that $\partial\Omega$ is represented by
\[
\partial\Omega=\{x\in\mathbb{R}^N\, :\, x=(1+\varphi(y))\,y, \mbox{ for } y\in\partial B_1\}.
\]
\end{definition}
\begin{definition}\label{def:hunmezzo}Given a function \(\varphi:\partial B_1\to \R\) we define 
\[
\|\varphi\|^2_{H^{1/2}(\partial B_1)}:=\int_{\partial B_1}\varphi^2 \,d\mathcal H^{N-1}+\int_{B_1} \left|\nabla  H(\varphi)\right|^2\, dx,
\]
where \( H(\varphi)\) is the $W^{1,2}$ harmonic extension of \(\varphi\), i.e.
\[
\Delta H(\varphi) =0\quad \text{in \(B_1\),}\qquad 
 H(\varphi)=\varphi \quad \text{on \(\partial B_1\).}
\]
\end{definition}
It can be easily proved  that the above norm is equivalent to the classical \(H^{1/2}\) norm and that \(H^{1/2}(\partial B_1)\) is a Hilbert space with this norm. Moreover, thanks to the following Poincar\'e-Wirtinger trace inequality (see for instance  \cite[Section 4]{BDP})
\[
\int_{\partial B_1} \left|w-\mean{\partial B_1} w\right|^2\, d\mathcal H^{N-1}\le \,\int_{B_1} |\nabla w|^2\, dx,\qquad w\in W^{1,2}(B_1),\]
we have 
\begin{equation}
\label{normequiv}
\|\nabla  H( \varphi)\|_{L^2(B_1)}\le \|\varphi\|_{H^{1/2}(\partial B_1)}\le \sqrt{2}\,\|\nabla   H( \varphi)\|_{L^2(B_1)},\quad \mbox{ for every } \, \varphi\,\mbox{ s.\,t. }\int_{\partial B_1} \varphi=0.
\end{equation}
The main result of this section is then the following, where we denote by
\begin{equation}
\label{defbar}
x_\Omega=\frac{1}{|\Omega|}\int_\Omega x\, dx,
\end{equation}
the barycenter of \(\Omega\).
\begin{theorem}
\label{prop:fine}
Let $0<\gamma\le 1$. Then there exists \(\delta_1=\delta_1(N,\gamma)\) such that if \(\Omega\) is a nearly spherical set of class $C^{2,\gamma}$ parametrized by $\varphi$ with 
\[
\|\varphi\|_{C^{2,\gamma}}\le \delta_1,\qquad |\Omega|=|B_1|\qquad \mbox{ and }\qquad x_\Omega=0,
\] 
then
\begin{equation}
\label{ballstab}
E(\Omega)-E(B_1)\ge \frac{1}{32\,N^2}\,\left\|\varphi\right\|^2_{H^{1/2}(\partial B_1)}.
\end{equation}
\end{theorem}
The proof of the above Theorem is based on the following Lemma, which is due to Dambrine, see \cite[Theorem 1]{Da}. For the sake of completeness  we give a sketch of its proof in Appendix \ref{sec:shape} at the end of the paper.
\begin{lemma}\label{dambrin}  Let $0<\gamma\le 1$, there exist a modulus of continuity \(\omega\) and a constant \(\delta_2=\delta_2(N,\gamma)\), such that, for every  \(C^{2,\gamma}\) nearly spherical set \(\Omega\) parametrized by \(\varphi\) with \(\|\varphi\|_{C^{2,\gamma}}\le \delta_2\) and \(|\Omega|=|B_1|\), we have
\begin{equation}\label{taylor}
E(\Omega)\ge E(B_1)+\frac 1 2\, \partial ^2 E(B_1)[\varphi,\varphi]-\omega \big(\|\varphi\|_{C^{2,\gamma}}\big)\,\|\varphi\|^2_{H^{1/2}(\partial B_1)},
\end{equation}
where, for every \(\varphi \in H^{1/2}(\partial B_1)\) we set
\begin{equation}
\label{hessian}
\partial ^2 E(B_1)[\varphi,\varphi]:=\frac {1}{N^2}\Big( \int_{ B_1} |\nabla H(\varphi)|^2\,dx-\int_{\partial B_1}\varphi^2\,d\mathcal H^{N-1}\Big).
\end{equation}
\end{lemma}
By using this result, we can now prove Theorem \ref{prop:fine}.
\begin{proof}[Proof of Theorem \ref{prop:fine}]
By assumption
\[
|\Omega|=\int_{\partial B_1} \frac{(1+\varphi)^N}{N}\, d\mathcal{H}^{N-1}=|B_1|\qquad\text{and}\qquad x_\Omega=
\int_{\partial B_1} y\, \frac{(1+\varphi)^{N+1}}{N+1}\, d\mathcal{H}^{N-1}=0.
\]
Thanks to  the smallness assumption on $\varphi$ we get 
\begin{equation}
\label{misuraperiferia}
\left|\int_{\partial B_1} \varphi\, d\mathcal{H}^{N-1}\right|=\left|\sum_{h=2}^N {N \choose h} \int_{\partial B_1} \frac{\varphi^h}{N}\, d\mathcal{H}^{N-1}\right|
\le C\, \int_{\partial B_1} \varphi^2\, d\mathcal{H}^{N-1}\le\, C\, \delta_1\, \|\varphi\|_{H^{1/2}(\partial B_1)},
\end{equation}
and
\begin{equation}
\label{bariperiferia}
\left|\int_{\partial B_1} y_i\, \varphi\, d\mathcal{H}^{N-1}\right|\le \sum_{h=2}^N {N \choose h} \int_{\partial B_1}\left|\frac{\varphi^h}{N+1}\right|\, d\mathcal{H}^{N-1}\le C\, \delta_1\, \|\varphi\|_{H^{1/2}(\partial B_1)},
\end{equation}
where \(C=C(N)\) is a dimensional constant. Thus we obtain that \(\varphi\) belongs to \(\mathcal M_{C\delta_1}\), where we define 
\[
\mathcal{M}_{\delta}:=\left\{\xi\in H^{1/2}(\partial B_1)\, :\, \left|\int_{\partial B_1} \xi\, d\mathcal{H}^{N-1}\right|+\left|\int_{\partial B_1} x\, \xi\, d\mathcal{H}^{N-1}\right|\le \delta\, \|\xi\|_{H^{1/2}(\partial B_1)}\right\}.
\]
By Lemma \ref{dambrin}, if $\delta_1\le \delta_2$ we can infer
\begin{equation}
\label{kebab}
E(\Omega)\ge E(B_1)+\frac 1 2 \partial ^2 E(B_1)[\varphi,\varphi]-\omega\big(\|\varphi\|_{C^{2,\gamma}}\big)\|\varphi\|^2_{H^{1/2}(\partial B_1)}.
\end{equation}
We now claim the following:   there exists \(\widehat \delta=\widehat \delta(N)>0\) such that if \(\delta\le\widehat \delta\) then
\begin{equation}
\label{conclude}
 \partial ^2 E(B_1)[\xi,\xi]\ge \frac{1}{8\,N^2}\|\xi\|^2_{H^{1/2}(\partial B_1)},\qquad \mbox{ for every } \, \xi \in \mathcal M_{\delta}.
\end{equation}
By choosing \(\delta_1\ll \min\{\widehat \delta,\delta_2\}\) sufficiently small it is clear that \eqref{conclude} together with \eqref{kebab} concludes the proof of \eqref{ballstab}. We are thus  left to prove \eqref{conclude}, which will be done in the  two steps below.
\vskip.2cm\noindent
\(\bullet\)\,{\it Step 1:} Let \(\mathcal M_0\) be 
\[
\mathcal{M}_0=\left\{\xi\in H^{1/2}(\partial B_1)\, :\, \int_{\partial B_1} \xi\, d\mathcal{H}^{N-1}=\int_{\partial B_1} x_i\, \xi\, d\mathcal{H}^{N-1}=0,\ i=1,\dots,N\right\},
\]
then
\begin{equation}\label{stek1}
 \partial ^2 E(B_1)[\xi,\xi]\ge \frac{1}{4N^2}\|\xi\|^2_{H^{1/2}(\partial B_1)}, \qquad\mbox{ for every } \xi\in \mathcal M_0.
\end{equation}
To see this, just notice that 
\[
\min\left\{\frac{\displaystyle\int_{B_1} |\nabla H(\xi)|^2\, dx}{\displaystyle \int_{\partial B_1} \xi^2\, d\mathcal{H}^{N-1}}\, :\, \xi\in\mathcal{M}_0\setminus\{0\}\right\}=2,
\]
in every dimension $N\ge 2$. Indeed the above minimum is the Rayleigh quotient of a Stekloff eigenvalue problem on \(\partial B_1\)  which has as associated eigenspace the  homogeneous  harmonic polynomials of degree $2$, see \cite[Section 4]{BDP} and \cite{Mu}. From this,  the definition of \(\partial ^2 E\) \eqref{hessian} and \eqref{normequiv} we get
\[
\partial ^2 E(B_1)[\xi,\xi]\ge \frac{1}{2N^2}\int_{B_1} \big|\nabla H(\xi)\big|^2\,dx\ge \frac{1}{4N^2}\|\xi\|^2_{H^{1/2}(\partial B_1)} \qquad\mbox{ for every } \xi\in \mathcal M_0,
\] 
which is \eqref{stek1}.
\vskip.2cm
\noindent 
\(\bullet\)\,{\it Step 2:} For every \(\xi\) in \(\mathcal M_{\delta}\) let us consider its $L^2$ projection on \(\mathcal M_0^{\perp}\), given by
\[
\Pi(\xi)=a_{0}\, Y_{0}+\sum_{i=1}^N a_{1,i}\, Y_{1,i},
\]
where
\[
Y_{0}=\sqrt{\frac{1}{N\,|B_1|}}\qquad \qquad\qquad Y_{1,i}(x)=\sqrt{\frac{1}{|B_1|}}\, x_i,\quad x\in\partial B_1,\quad i=1,\dots,N.
\]
and
\[
a_{0}=\int_{\partial B_1} \xi\,Y_0\, d\mathcal{H}^{N-1}\qquad \qquad\qquad a_{1,i}=\int_{\partial B_1} \xi\,Y_{1,i}\, d\mathcal{H}^{N-1},\quad i=1,\dots,N.
\]
It is immediate to check that \(\xi-\Pi(\xi)\in \mathcal M_0\). Moreover by Green formula
\begin{equation}
\label{sigaretta1}
\|\xi-\Pi(\xi)\|^2_{H^{1/2}(\partial B_1)}=\, \|\xi\|^2_{H^{1/2}(\partial B_1)}-\|\Pi(\xi)\|^2_{H^{1/2}(\partial B_1)},
\end{equation}
and, by the definition of \(\mathcal M_{\delta}\),
\begin{equation}
\label{sigaretta2}
\|\Pi(\xi)\|^2_{H^{1/2}(\partial B_1)}=a_{0}^2+2\, \sum_{i=1}^N a_{1,i}^2\le C\, \delta^2\, \|\xi\|^2_{H^{1/2}(\partial B_1)}.
\end{equation}
By bilinearity and Step 1, we have
\begin{equation}
\label{odiolosplit}
\begin{split}
\partial^2 E[\xi,\xi]&=\partial^2 E[\xi-\Pi(\xi),\xi-\Pi(\xi)]+2\,\partial^2 E[\xi,\Pi(\xi)]-\partial^2 E[\Pi(\xi),\Pi(\xi)]\\
&\ge \frac{1}{4N^2}\|\xi-\Pi(\xi)\|^2_{H^{1/2}(\partial B_1)} -2\|\xi\|_{H^{1/2}(\partial B_1)}\|\Pi(\xi)\|_{H^{1/2}(\partial B_1)}-\|\Pi(\xi)\|_{H^{1/2}(\partial B_1)}^2,
\end{split}
\end{equation}
where we have used the trivial estimate
\[
\left|\partial^2 E(B_1)[\xi_1,\xi_2]\right|\le \|\xi_1\|_{H^{1/2}(\partial B_1)}\|\xi_2\|_{H^{1/2}(\partial B_1)}\qquad\mbox{ for every } \,\xi_1,\, \xi_2 \in H^{1/2}(\partial B_1).
\]
Equation \eqref{odiolosplit}, together with \eqref{sigaretta1} and \eqref{sigaretta2},  gives 
\[
\partial^2 E[\xi,\xi]\ge \frac{1}{4N^2}\|\xi\|_{H^{1/2}(\partial B_1)}^2-C\,\delta\,\|\xi\|_{H^{1/2}(\partial B_1)}^2, 
\]
from which \eqref{conclude} follows, choosing \(\widehat \delta\) small enough.
\end{proof}

\section{Third step: stability for bounded sets with small asymmetry} 

Throughout the rest of the paper we will denote by $B_R(x_0)$ the ball
\[
B_R(x_0)=\{x\in\mathbb{R}^N\, :\, |x-x_0|<R\}.
\] 
When $x_0$ coincides with the origin, we will simply use the notation $B_R$.
\vskip.2cm
\subsection{Stability via a selection principle} The aim of this section is to prove the validity of the quantitative Saint-Venant inequality for bounded sets with small asymmetry. For this, we need to replace the Fraenkel asymmetry \(\A(\Omega)\) with a smoother asymmetry functional, as explained in the Introduction.
\begin{definition} 
Given a bounded set \(\Omega\subset\mathbb{R}^N\), we define
\begin{equation}\label{defalpha}
\alpha(\Omega)=\int_{\Omega\Delta B_1(x_\Omega)} \big|1-|x-x_\Omega|\big|\, dx,
\end{equation}
where \(x_\Omega\) is the barycenter of \(\Omega\) introduced in  \eqref{defbar}.
Notice that \(\alpha(\Omega)=0\) if and only if \(\Omega\) is a ball of radius \(1\), moreover we can write 
\begin{equation}
\label{alphacomodo}
\alpha(\Omega)=\beta_N+\int_{\Omega} (|x-x_\Omega|-1)\, dx,\quad \mbox{ where }\quad\beta_N=\int_{B_1}\big(1-|x|\big)\, dx.
\end{equation}
\end{definition}Below we summarize the main properties of \(\alpha\).
\begin{lemma}\label{lem:alphaprop}Let \(R\ge 2\), then: 
\begin{enumerate}
\item[(i)] there exists a constant \(C_1=C_1(N)\) such that for every \(\Omega\)
\[
C_1\,\alpha(\Omega)\ge |\Omega\Delta B_1(x_\Omega)|^2;
\]
\item[(ii)] there exists a constant \(C_2=C_2(R)\) such that for every \(\Omega_1, \Omega_2\subset B_R\), we have
\[
|\alpha(\Omega_1)-\alpha(\Omega_2)|\le C_2\,|\Omega_1\Delta \Omega_2|;
\]
\item[(iii)] there exists two constants $\delta_3=\delta_3(N)>0$ and $C_{3}=C_3(N)>0$ such that for every nearly spherical set $\Omega$ with $\|\varphi\|_{L^\infty}\le \delta_3$, we have
\[
\alpha(\Omega)\le C_{3}\, \|\varphi\|^2_{L^2(\partial B_1)}.
\]
\end{enumerate}
\end{lemma}

\begin{proof}The proof of (i) can be obtained by a simple rearrangement argument, similar to that used in the proof of \cite[Theorem 2.2]{BDP}. First of all, we can suppose for simplicity that $x_\Omega=0$, then 
\begin{equation}
\label{alpha}
\alpha(\Omega)=\int_{\Omega\setminus B_1} (|x|-1)\, dx+\int_{B_1\setminus\Omega} (1-|x|)\, dx.
\end{equation}
We then introduce the annular regions
\[
T_1=\{x\in\mathbb{R}^N\, :\, 1<|x|<R_1\}\qquad \mbox{ and }\qquad T_2=\{x\in\mathbb{R}^N\, :\, R_2<|x|<1\},
\]
where the two radii $R_1$ and $R_2$ are such that $|T_1|=|\Omega\setminus B_1|$ and $|T_2|=|B_1\setminus\Omega|$, i.e.
\[
R_1=\left(1+\frac{|\Omega\setminus B_1|}{|B_1|}\right)^\frac{1}{N}\qquad \mbox{ and }\qquad R_2=\left(1-\frac{|B_1\setminus\Omega|}{|B_1|}\right)^\frac{1}{N}.
\] 
By using this and the fact that in \eqref{alpha} we are integrating two monotone functions of the modulus, we get
\[
\begin{split}
\alpha(\Omega)&\ge \int_{T_1} (|x|-1)\, dx+\int_{T_2} (1-|x|)\, dx\\
&=\omega_N\, \left[\frac{R_1^{N+1}-1}{N+1}-\frac{R_1^N-1}{N}+\frac{R_2^{N+1}-1}{N+1}-\frac{R_2^N-1}{N}\right]\ge \frac{1}{C_1}\,|\Omega\Delta B_1|^2.
\end{split}
\]
In order to prove (ii), we first notice that by using \eqref{alphacomodo} and triangular inequality, we get
\[
\begin{split}
|\alpha(\Omega_1)-\alpha(\Omega_2)|&\le \left|\int_{\Omega_1} |x-x_{\Omega_1}|\, dx-\int_{\Omega_2}|x-x_{\Omega_2}|\, dx \right|+ |\Omega_1\Delta \Omega_2|\\
&\le\int_{\Omega_1\cap\Omega_2} |x_{\Omega_1}-x_{\Omega_2}|\, dx+\int_{\Omega_1\setminus\Omega_2} |x-x_{\Omega_1}|\, dx\\
&+\int_{\Omega_2\setminus\Omega_1} |x-x_{\Omega_2}|\, dx+|\Omega_1\Delta \Omega_2|.
\end{split}
\]
Finally, by using that
\[
|\Omega_1\cap\Omega_2|\,|x_{\Omega_1}-x_{\Omega_2}|\le C(R)\,|\Omega_1\Delta \Omega_2|,
\]
and that \(|x-x_\Omega|\le 2R\) for every \(x\in B_R\), we can conclude.
\vskip.2cm\noindent
We then prove property (iii), for nearly spherical sets. By definition of $\alpha(\Omega)$
\[
\begin{split}
\alpha(\Omega)&=\int_{\Omega\setminus B_1} (|x|-1)\, dx+\int_{B_1\setminus \Omega} (1-|x|)\, dx\\
&=\int_{\{\varphi\ge 0\}} \frac{(1+\varphi(y))^{N+1}-1}{N+1}\, d\mathcal{H}^{N-1}-\int_{\{\varphi\ge 0\}} \frac{(1+\varphi(y))^{N}-1}{N}\, d\mathcal{H}^{N-1}\\
&+\int_{\{\varphi< 0\}} \frac{(1+\varphi(y))^{N+1}-1}{N+1}\, d\mathcal{H}^{N-1}-\int_{\{\varphi< 0\}} \frac{(1+\varphi(y))^{N}-1}{N}\, d\mathcal{H}^{N-1}.
\end{split}
\]
By observing that
\[
\frac{(1+t)^{N+1}-1}{N+1}\le t+\frac{N}{2}\, t^2,\quad t\in\mathbb{R}\qquad\mbox{ and }\qquad \frac{(1+t)^{N}-1}{N}\ge t+\frac{N-1}{4}\, t^2,\quad |t|\le \frac{3}{2\, (N-2)},
\]
we obtain the estimate.
\end{proof}
This is the main result of this section.
\begin{theorem}
\label{thm:stability_lim}
For every $R\ge 2$, there exist two constants $\widehat\sigma=\widehat \sigma(N,R)>0$ and $\widehat\varepsilon=\widehat \varepsilon(N,R)>0$ such that
\begin{equation}
\label{stability_lim}
E(\Omega)- E(B_1)\ge \widehat \sigma\,\alpha(\Omega), 
\end{equation}
for all sets \(\Omega\) contained in \(B_{R}\) with $|\Omega|=|B_1|$ and $\alpha(\Omega)\le \widehat \varepsilon$. 
\end{theorem}
In order to prove Theorem \ref{thm:stability_lim}, we argue by contradiction. Up to rename \(\sigma\), we assume  that there exists a sequence of sets \(\widetilde{\Omega}_j\subset B_R\) such that 
\begin{equation}
\label{contradiction}
|\widetilde\Omega_j|=|B_1|,\qquad\eps_j:=\alpha(\widetilde{\Omega}_j)\to 0 \qquad  \text{  while } \qquad
E(\widetilde{\Omega}_j)-E(B_1)\le \sigma^4 \eps_j,
\end{equation}
where \(\sigma<1\) is a suitably small parameter that will be chosen later\footnote{We put \(\sigma^4\) just to simplify some of the computations below.}. The key ingredient is given by the following.
\begin{prop}[Selection Principle]
\label{thm:sel}
Let $R\ge 2$ then there exists $\widetilde\sigma=\widetilde \sigma(N,R)>0$ such that if \(\sigma \le \widetilde\sigma(N,R)\) and \(\widetilde{\Omega}_j\) are as in \eqref{contradiction}, then we can find a sequence of smooth open sets \(U_j\subset B_R\) satisfying: 
\begin{enumerate}
\item[(i)] \(|U_j|=|B_1|\); 
\item[(ii)] \(x_{U_j}=0\);
\item[(iii)] \(\partial U_j\) are converging to \(\partial B_1\) in \(C^k\) for every \(k\);
\item[(iv)] there holds
\begin{equation}
\label{disfottuta}
\limsup_{j \to \infty}\frac{E(U_j)-E(B_1)}{\alpha(U_j)}\le \widetilde{C}\,\sigma,
\end{equation}
for some constant \(\widetilde C=\widetilde C(N,R)\).
\end{enumerate}
\end{prop}
The proof of the Selection Principle is quite involved and will occupy the rest of the section. By combining this result and the stability estimate for nearly spherical sets, we can conclude the proof of Theorem \ref{thm:stability_lim}.
\begin{proof}[Proof of Theorem \ref{thm:stability_lim}] 
As above, arguing by contradiction we can exhibit a sequence of sets $\{U_j\}$ smoothly converging to the ball $B_1$ and having the properties expressed by Proposition \ref{thm:sel}. In particular, for $j\in\mathbb{N}$ large enough each $U_j$ is a nearly spherical set of class $C^{2,\gamma}$, satisfying the hypotheses of Theorem \ref{prop:fine}. The latter, Lemma \ref{lem:alphaprop} (iii) and Proposition \ref{thm:sel} (iv) then give 
\[
\frac{1}{32N^2C_{3}}\le \limsup_{j\to\infty} \frac{E(U_j)-E(B_1)}{\alpha(U_j)}\le \widetilde{C}\, \sigma.
\]
By choosing $\sigma$ suitably small, we get the desired contradiction. 
\end{proof}

\subsection{Proof of the Selection Principle: a penalized minimum problem} 
In order to prove Proposition \ref{thm:sel} above, we would like to use the local regularity theory for free boundary-type problem. As explained in the Introduction,  we need to get rid of the volume constraint $|\Omega|=|B_1|$. To this end we introduce the following function (see \cite{AAC})
\[
f_\eta(s)=
\begin{cases}
\eta(s-\omega_N)\quad &\text{if \(s\le \omega_N\)},\\
(s-\omega_N)/\eta&\text{if \(s\ge \omega_N\)}.
\end{cases}
\]
Notice that the function \(f_\eta\) defined above satisfies the following key property
\begin{equation}
\label{f-prop}
\eta\,(s_1-s_2)\le f_\eta(s_1)-f_\eta(s_2)\le \frac 1 \eta\, (s_1-s_2),
\end{equation}
for every \(0\le s_2\le s_1\).
\begin{lemma}
\label{torsionelimitato}
For every \(R\ge 2\) there exists a \(\widehat\eta=\widehat \eta(R)\) such that, up to translation, \(B_1\) is a minimizer of
\begin{equation}\label{Feta}
\F_{\widehat\eta}(\Omega)=E(\Omega)+f_{\widehat\eta}(|\Omega|),
\end{equation}
among all sets contained in $B_R$. Moreover, there exists a costant $C_4=C_4(N,R)>0$ such that for any other ball $B_r$ with $0\le r\le R$, there holds
\begin{equation}
\label{eq:coercivoraggi}
\F_{\widehat\eta }(B_r)-\F_{\widehat \eta }(B_1)\ge  \frac{|r-1|}{C_4}.
\end{equation}
\end{lemma}
\begin{proof}By using the P\'olya-Szeg\H{o} principle \eqref{PS} it is easily seen that among minimizers of  \(\F_\eta\)  there is a ball of radius \(r(\eta)\le R\). Let us show that we can choose \(\eta\) such that \(r=1\). To  this aim, we introduce
\[
g(r)=\F_\eta(B_r)=r^{N+2}E(B_1)+f_\eta(\omega_N r^N).
\]
Assume that \(1< r\le R\), then
\[
g'(r)=r^{N-1}\,\left((N+2)\, r^2\, E(B_1)+\frac{N\omega_N}{\eta}\right)\ge r^{N-1}\left(-(N+2)\,R^2\, |E(B_1)|+\frac{N\omega_N}{\eta}\right)>0,
\]
if \(\eta\) is small enough.  For \(r\le 1\) we notice that we can easily choose \(\eta\ll1\) such that
\[
r\mapsto g(r),\qquad 0<r\le 1,
\]
admits a minimum in \(r=1\). Moreover it is easy to see that with the above choice of \(\eta\) there exists a constant \(C=C(N,R)\) such that
\[
\lim_{r\to 1^-}g'(r)\le -1/C\quad \lim_{r\to 1^+} g'(r)\ge 1/C,
\] 
from which \eqref{eq:coercivoraggi} follows.
\end{proof}

Up to a translation and a (small) dilation the sets \(U_j\) constructed in Proposition \ref{thm:sel} are given by the family of  minimizers of the following penalized problems
\begin{equation}
\label{mingj}
\min\Big\{\mathcal{G}_{\widehat\eta,j}(\Omega)\ :\  \Omega\subset B_R \Big\}.
\end{equation}
Here the functionals $\mathcal{G}_{\widehat\eta,j}$ are given by
\[
\G_{\widehat\eta,j}(\Omega)=\F_{\widehat \eta} (\Omega)+\sqrt{\eps_j^2+\sigma^2(\alpha(\Omega)-\eps_j)^2}.
\]
Following a by now classical approach, in order to find a minimizer to \eqref{mingj}, we need to extend the functionals \(\G_{\widehat \eta,j}\) to the class of {\em quasi-open } sets.  Referring to \cite[Chapter 4]{BuBu} for a complete account on the theory of these sets, we simply recall here the main facts needed in the sequel. 

A Borel set \(U\) is said {\em quasi-open} if there is a \(W^{1,2}(\mathbb{R}^N)\) function \(u\) such that
\[
U=\{x\, :\, \widetilde u(x)>0\}, 
\]
where \(\widetilde u\) is the precise representative of \(u\),  uniquely defined outside a set of zero capacity, see \cite[Section 4.8]{EG}. Given a quasi-open set \(U\) we can define 
\[
W^{1,2}_0(U)=\Big\{ v\in W^{1,2}(\mathbb{R}^N):\ {\rm Capacity} \left(\{v\ne 0\}\cap (\mathbb{R}^N\setminus U)\right)=0\Big\},
\]
which is a strongly closed and convex subset of \(W^{1,2}\) (hence also weakly closed). Then for a quasi-open set $U$ its \emph{energy} is still defined as
\begin{equation}\label{energyquasiopen}
E(U)=\inf_{v\in W^{1,2}_0(U)} \frac 1  2 \int_U |\nabla v|^2\, dx-\int_U v \, dx.
\end{equation}
The function \(u_U\) achieving the above infimum is still called the \emph{energy function} of \(U\). The following ``minimum principle'' is easily seen to holds true
\begin{equation}
\label{maxprin}
U=\{x\, :\, \widetilde u_U(x)>0\}. 
\end{equation}
We are now ready to prove the following.
\begin{lemma}
\label{existence}
There exists \(\sigma_1=\sigma_1(N,R)>0\) such that if \(\sigma \le \sigma_1\) then the infimum \eqref{mingj} is attained by a quasi-open set \(\Omega_j\). Moreover the perimeter of  \(\Omega_j\) is bounded independently on \(j\).
\end{lemma}
\begin{proof} 
Let \(\{\mathcal{O}_k\}_{k\in\mathbb{N}}\subset B_R\) be a minimizing sequence satisfying
\[
\G_{\widehat\eta,j}(\mathcal{O}_k)\le \inf \G_{\widehat\eta,j}+\frac{1}{k},\qquad k\in\mathbb{N}.
\]
Denoting with  \(u_k=\widetilde u_{\mathcal{O}_k}\) the precise representative of the  energy function of \(\mathcal{O}_k\), \eqref{maxprin} yields
\begin{equation}\label{omegak}
\mathcal{O}_k=\{x\, :\, u_k(x)>0\}.
\end{equation} 
Let us set $t_k=1/\sqrt{k}$, then we define
\[
V_k=\{x\, :\, u_k(x)>t_k\}.
\]
Notice that the function \(v_k=(u_k-t_k)_+\) is the energy function for \(V_k\).
By this and by 
\[
\G_{\widehat\eta,j}(\mathcal{O}_k)\le \G_{\widehat\eta,j}(V_k)+1/k,
\]
we infer
\[
\begin{split}
\frac 1  2 \int |\nabla u_k|^2\, dx&-\int u_k\, dx +f_{\widehat\eta}(|\{u_k>0\}|)+\sqrt{\eps_j^2+\sigma^2\,(\alpha(\{u_k>0\})-\eps_j)^2}\\
&\le  \frac 1  2 \int_{\{u_k>t_k\}} |\nabla u_k|^2\, dx-\int_{\{u_k>t_k\}} (u_k-t_k)_+\, dx\\
&+f_{\widehat\eta}(|\{u_k>t_k\}|)+\sqrt{\eps_j^2+\sigma^2\,(\alpha(\{u_k>t_k\})-\eps_j)^2}+\frac 1 k.
\end{split}
\]
Using \eqref{f-prop}, the Lipschitz character of the function $t\mapsto \sqrt{\varepsilon_j^2+\sigma^2\, (t-\varepsilon_j)^2}$ and Lemma \ref{lem:alphaprop} (ii) we obtain
\[
\begin{split}
\frac 1 2 \int_{\{0<u_k<t_k\}}|\nabla u_k|^2\, dx+\widehat \eta\, |\{0<u_k<t_k\}|&\le t_k\, |\{u_k>0\}|+\sigma\, |\alpha(\{u_k>0\})-\alpha(\{u_k>t_k\})|+\frac 1 k \\
&\le t_k\, |\{u_k>0\}|+C_2\,\sigma\, |\{0<u_k<t_k\}|+\frac 1 k.
\end{split}
\]
Choosing \(\sigma\) such that $C_2\,\sigma\le \widehat\eta/2$ we obtain
\[
\frac 1 2 \int_{\{0<u_k<t_k\}}|\nabla u_k|^2\, dx+\frac{\widehat \eta}{2}\, |\{0<u_k<t_k\}|\le |B_R|\, t_k +\frac{1}{k}.
\]
By co-area formula, Cauchy-Schwarz inequality and recalling that \(\widehat \eta <1\), we infer
\[
\begin{split}
\widehat\eta \int_{0}^{t_k}P(\{u_k>s\})\,ds &=  \widehat \eta \int_{\{0<u_k<t_k\}} |\nabla u_k|\, dx\\
 &\le \frac {\widehat \eta} {2} \int_{\{0<u_k<t_k\}}|\nabla u_k|^2\, dx+\frac{\widehat \eta}{2}\, |\{0<u_k<t_k\}|\\
 &\le |B_R|\, t_k +\frac {1}{k}.
\end{split}
\]
By recalling that \(t_k =1/\sqrt{k}\), we can find a level \(0\le s_k\le1/\sqrt{k}\) such that the sets \[W_k=\{x\, :\, u_k(x)>s_k\},\] satisfy
\begin{equation}\label{perimeter}
\begin{split}
P(W_k)\le \frac{2\,\widehat\eta}{\widehat\eta\, t_k}\int_{0}^{t_k}P(\{u_k>s\})\,ds\le \frac{2|B_R|}{\widehat \eta}+\frac{2}{\widehat\eta\, t_k k}= C(N,R)+\frac{2}{\widehat\eta\sqrt k}.
\end{split}
\end{equation}
 We claim that \(W_k\) is still a minimizing sequence. Indeed, using \eqref{f-prop} and Lemma \ref{lem:alphaprop} (ii), for \(\sigma\) such that $C_2\,\sigma\le\widehat\eta/2 $ we have
\begin{equation}
\label{serve}
\begin{split}
\G_{\widehat\eta,j}(W_k)&=\int_{\{u_k>s_k\}}|\nabla u_k|^2\, dx-\int_{\{u_k>s_k\}} (u_k-s_k)_+\, dx\\
&\quad +f_{\widehat \eta}(|\{u_k>s_k\}|)+\sqrt{\eps_j^2+\sigma^2\,(\alpha(\{u_k>s_k\})-\eps_j)^2}\\
&\le \G_{\widehat\eta,j}(\mathcal{O}_k)+s_k\,|\{u_k>0\}|+ f_{\widehat \eta}(|\{u_k>s_k\}|)-f_{\widehat \eta}(|\{u_k>0\}|)\\
&\quad+\sigma\, \big|\alpha(\{u_k>s_k\})-\alpha(\{u_k>0\})\big|\\
&\le \G_{\widehat\eta,j}(\mathcal{O}_k)+|B_R|/\sqrt{k}-(\widehat \eta-C_2\,\sigma)\,  |\{0<u_k<s_k\}|\le \G_{\widehat\eta,j}(\mathcal{O}_k)+|B_R|/\sqrt{k},
\end{split}
\end{equation}
where we used again that $(u_k-s_k)_+$ is the energy function of $W_k$.
\par
By compactness of sets with equi-bounded perimeter, \eqref{perimeter} implies the existence of a Borel set \(W_\infty\) such that 
 \[
 1_{W_k} \to1_{ W_\infty}\ \text{ in \(L^1(B_R)\)}\qquad \text{and}\qquad P(W_\infty)\le C(N,R).
 \]
 Setting $w_k=(u_k-s_k)_+$, it is immediate to see that this is an equi-bounded sequence in $W^{1,2}_0(B_R)$, thus up to subsequences we can infer the existence of $w\in W^{1,2}_0(B_R)$ such that 
\[
\lim_{k\to\infty}\| w_k-w\|_{L^2}=0.
\]
If we set \(W=\{x\, :\, \widetilde w(x)>0\}\), then 
\[
1_{W}(x)\le \liminf_{k\to\infty} 1_{W_k}(x)=1_{W_\infty}(x),\qquad \mbox{ for a.e. }x\in B_R,
\]
which implies \(|W\setminus W_\infty|=0\). By the semicontinuity of the Dirichlet integral and  the continuity of \(\alpha(\cdot)\) with respect to the \(L^1\) convergence of sets, passing to the limit as \(k\) goes to \(\infty\) in \eqref{serve}, we get
\[
\begin{split}
E(W) +f_{\widehat \eta} (|W_\infty|)&+\sqrt{\eps_j^2+\sigma^2(\alpha(W_\infty)-\eps_j)^2}\\
&\le \frac 1 2 \int |\nabla w|^2\,dx-\int w\,dx+ f_{\widehat \eta} (|W_\infty|)+\sqrt{\eps_j^2+\sigma^2(\alpha(W_\infty)-\eps_j)^2}\\
&\le \inf \G_{\widehat\eta,j} \le E(W)+f_{\widehat \eta} (|W|)+\sqrt{\eps_j^2+\sigma^2(\alpha(W)-\eps_j)^2}. 
\end{split}
\]
This in turn gives 
\[
f_{\widehat\eta}(|W_\infty|)-f_{\widehat\eta}(|W|)\le \sigma\, |\alpha(W)-\alpha(W_\infty)|
\]
which together with Lemma \ref{lem:alphaprop} (ii), \eqref{f-prop} and $|W\setminus W_\infty|=0$ yields
\[
\widehat \eta\, |W_\infty\setminus W|\le C_2\,\sigma\, |W_\infty\setminus W|.
\]
Since $C_2\, \sigma\le \widehat\eta/2$, this implies that \(|W\Delta W_\infty|=0\), so that \(W\) is the desired minimizer $\Omega_j$.
\end{proof}

\subsection{Proof of the Selection Principle: properties of the minimizers}

\begin{lemma} [Properties of  minimizers, Part I]\label{lm:prop1} The sequence of minimizers \(\{\Omega_j\}_{j}\) found in Lemma \ref{existence} satisfies the following properties:
\begin{itemize}
\item[(i)] \(|\alpha(\Omega_j)-\eps_j|\le 3\,\sigma \,\eps_j\) and \(\big||\Omega_j|-|B_1|\big|\le C_5\,\sigma^4\, \eps_j\), where \(C_5=C_5(R,N)\);
\item[(ii)] up to translations \(\Omega_j\to B_1\) in \(L^1\);
\item [(iii)] the following inequality holds true
\begin{equation}
\label{deficit}
0\le \F_{\widehat\eta}(\Omega_j)-\F_{\widehat\eta }(B_1)\le \sigma^4 \eps_j.
\end{equation}
\end{itemize}
\end{lemma}
\begin{proof}
We start noticing that by the minimality property of \(\Omega_j\) and by the definition  \eqref{contradiction} of $\widetilde\Omega_j$
\begin{equation}\label{1}
\begin{split}
\F_{\widehat \eta}(\Omega_j)+\eps_j&\le \F_{\widehat \eta}(\Omega_j)+\sqrt{\eps_j^2+\sigma^2(\alpha(\Omega_j)-\eps_j)^2}\\
&=\G_{\widehat\eta,j}(\Omega_j)\le \G_{\widehat\eta,j}(\widetilde \Omega_j)=\F_{\widehat \eta}(\widetilde\Omega_j)+\eps_j \le \F_{\widehat \eta}(B_1)+(1+\sigma^4)\,\eps_j,
\end{split}
\end{equation}
from which we obtain \eqref{deficit}. Moreover, since \(B_1\) minimizes \(\F_{\widehat \eta}\), from the previous we deduce that
\[
\sqrt{\eps_j^2+\sigma^2(\alpha(\Omega_j)-\eps_j)^2}\le \eps_j\,(1+\sigma^4), 
\]
which implies, since \(\sigma<1\),
\[
\eps_j^2+\sigma^2(\alpha(\Omega_j)-\eps_j)^2\le \eps^2_j\, (1+\sigma^4)^2\le \eps_j^2\,(1+3\,\sigma^4).
\]
From this we obtain the first part of point (i). To obtain the second we notice that if \(B_{\Omega_j}\) is a ball of the same measure as \(\Omega_j\), then by the P\`olya-Szeg\H{o} principle
\[
\F_{\widehat\eta}(B_{\Omega_j})\le \F_{\widehat \eta}(\Omega_j),
\]
hence, by \eqref{deficit} and \eqref{eq:coercivoraggi},
\[
\left||\Omega_j|^{1/N}-|B_1|^{1/N}\right|\le C_4\,\omega_N^{1/N}\,\sigma^4 \eps_j.
\]
To prove point (ii) we notice that, up to translations, we can assume that \(x_\Omega=0\). By Lemma \ref{existence} the sets  \(\Omega_j\) have equi-bounded perimeter hence they are pre-compact in \(L^1(B_R)\). By the continuity of \(\alpha(\cdot)\), with respect to the \(L^1\) convergence, and point (i) we see that any limit set \(\Omega_\infty\) satisfies \(\alpha(\Omega_\infty)=0\), from which point (ii) follows.
\end{proof}
\vskip.2cm
We now start studying the regularity of the sets \(\Omega_j\). In order to do this we recall that by \eqref{maxprin} \(\Omega_j=\{u_j>0\}\) where \(u_j=u_{\Omega_j}\) is the energy function of \(\Omega_j\). If \(v\in W^{1,2}_0(B_R)\), testing the minimality of \(\Omega_j\) with \(\{v>0\}\) and recalling the definition of energy \eqref{energyquasiopen}, we immediately see that \(u_j\) satisfies the following minimum property
\begin{equation}\label{umin}
\begin{split}
\frac 1 2 \int |\nabla u_j|^2\, dx&-\int u_j\, dx+f_{\widehat\eta}(|\{u_j>0\}|)+\sqrt{\eps_j^2+\sigma^2\,(\alpha(\{u_j>0\})-\eps_j)^2}\\
&\le \frac 1 2  \int |\nabla v|^2\, dx-\int v\, dx+f_{\widehat \eta}(|\{v>0\}|)+\sqrt{\eps_j^2+\sigma^2\,(\alpha(\{v>0\})-\eps_j)^2}.\\
\end{split}
\end{equation}
Using Lemma \ref{lem:alphaprop}, 
we obtain that \(u_j\) behaves like a perturbed minimum of the free boundary-type problem, more precisely
\begin{equation}\label{quasimin}
\begin{split}
\frac 1 2 \int |\nabla u_j|^2\, dx&-\int u_j\, dx+f_{\widehat\eta}(|\{u_j>0\}|)\\ &\le \frac 1 2  \int |\nabla v|^2\, dx-\int v\, dx+f_{\widehat \eta}(|\{v>0\}|)+C_2\,\sigma\, \big|\{u_j>0\}\Delta \{v>0\}\big|,
\end{split}
\end{equation}
for all \(v\in W^{1,2}_0(B_R)\).
\begin{oss}
The above two equations are the starting point to study  the regularity of \(\partial \Omega_j=\partial \{u_j>0\}\) using the techniques of Alt and Caffarelli, \cite{AC}.  We remark that \eqref{quasimin} can be summarized by saying the \(u_j\) is a \emph{quasi-minimizer} of the free boundary problem, in the spirit of \emph{perimeter quasi-minimizers}, see \cite[Part 3]{Maggi}. However in this kind of problems this notion can not provide too much regularity of \(\partial\{u_j>0\}\), indeed in general the volume term appearing in the right-hand side of \eqref{quasimin} is not lower order. To obtain our results we have to take advantage that the parameter \(\sigma\) multiplying such a term can be taken much smaller 
than \(\widehat \eta\). 
\end{oss}
After \cite{AC} it is by now well understood that the first step in order to prove regularity for solutions of \eqref{umin} is to show that 
\[
u_j(x)\sim \mathrm{dist}\left(x,\partial\{u_j>0\}\right), \qquad x\in\{u_j>0\},
\]
in some integral sense. This will be done in the next two Lemmas, which are the analogous of \cite[Lemma 3.4]{AC} and \cite[Lemma 3.2]{AC}.
\begin{lemma}
\label{lemma1}Let \(u_j\) be as above. There exists \(\sigma_2=\sigma_2(N,R)>0\) such that for every \(\kappa\in(0,1)\) one can find positive constants \(m\), \(\varrho_0\) depending only on \(\kappa\), \(R\) and the dimension,  such that, if \(\sigma\le \sigma_2\),  \(\varrho\le \varrho_0\), \(x_0\in B_R\) and 
\[
\mean{\partial B_{\varrho}(x_0)\cap B_R}u_j\, d\mathcal{H}^{N-1}\le m\, \varrho,
\]
then \(u=0\) in \(B_{\kappa \varrho}(x_0)\cap B_R\).
\end{lemma}

\begin{proof} Being \(j\) fixed for notational simplicity we drop the subscript. Morever, being \(x_0\) fixed we simply write \(B_{\varrho}\) for \(B_{\varrho}(x_0)\). 
\par
It is well known that \(u\) (extended to \(0\) outside \(B_R\)) satisfies \(-\Delta u\le 1\) in the weak sense, hence the function
\[
u(x)+\frac{|x-x_0|^2-\varrho^2}{2N},
\]
is subharmonic in \(B_\varrho(x_0)\). Therefore for every \(\kappa \in (0,1)\) there is \(C=C(\kappa,N)\) such that
\begin{equation}
\label{sup}
\begin{split}
\delta_\varrho:=\sup_{B_{\sqrt{\kappa}\varrho}} u &\le C\,\Big( \mean{\partial B_{\varrho}} u\,d\mathcal{H}^{N-1} +\varrho^2\Big)\le C\,(m\varrho+\varrho^2). 
\end{split}
\end{equation}
Let \(w\) be the solution of 
\[
\begin{cases}
-\Delta w=1&\text{in \(B_{\sqrt \kappa \varrho}\setminus B_{ \kappa \varrho}\)},\\
w=\delta_\varrho  &\text{on \(\partial B_{\sqrt \kappa \varrho}\)},\\
w=0 &\text{on \(B_{\kappa \varrho}\)}.
\end{cases}
\]
Since \(w\ge u\) on \(\partial B_{\sqrt{\kappa}\varrho}\), the function 
\[
v=
\begin{cases}
u\quad &\text{on \(\R^N\setminus B_{\sqrt{\kappa}\varrho}\)}\\
\min\{u, w\}&\text{on \(B_{\sqrt{\kappa}\varrho}\)},
\end{cases}
\]
satisfies
\[
\{v>0\} \subset \{u>0\}\qquad \mbox{ and }\qquad \{v>0\}\setminus B_{\sqrt{\kappa}\,\varrho}=\{u>0\}\setminus B_{\sqrt{\kappa}\,\varrho}.
\] 
In particular $v\in W^{1,2}_0(B_R)$ and \eqref{quasimin} gives
\[
\begin{split}
&\frac 1 2  \int_{B_{\sqrt \kappa \varrho}} |\nabla u|^2\, dx-\int_{B_{\sqrt \kappa \varrho}}u\, dx +f_{\widehat\eta} (|\{u>0\}|)\\
&\le 
\frac 1 2 \int_{B_{\sqrt \kappa \varrho}} |\nabla v|^2\, dx-\int_{B_{\sqrt \kappa \varrho}}v\, dx +f_{\widehat \eta} (|\{v>0\}|)+C_2\,\sigma\, \big|(\{u>0\}\setminus \{v>0\})\cap B_{\sqrt \kappa  \varrho}\big|.
\end{split}
\]
Since $v=0$ in $B_{\kappa\varrho}$,  
\[
|\{u>0\}\cap B_{ \kappa \varrho}|\le |(\{u>0\}\setminus \{v>0\})\cap B_{\sqrt \kappa \varrho}|.
\]
Using \eqref{f-prop} and choosing $\sigma>0$ such that \(C_2\,\sigma \le \widehat \eta /2\), the above two equations and the definition of $v$ give
\begin{equation}
\label{freddissimo}
\begin{split}
\frac 1 2  \int_{B_{ \kappa \varrho}}& |\nabla u|^2\, dx-\int_{B_{\kappa\varrho}} u\, dx + \frac{\widehat\eta}{2}\,|\{u>0\}\cap B_{ \kappa \varrho}|\\
&\le \frac 1 2  \int_{B_{ \kappa \varrho}} |\nabla u|^2\, dx-\int_{B_{\kappa\varrho}} u\, dx + \frac{\widehat \eta}{2}\,|(\{u>0\}\setminus \{v>0\})\cap B_{\sqrt \kappa \varrho}|\\
&\le \frac 1 2  \int_{B_{\sqrt \kappa \varrho}\setminus B_{ \kappa \varrho}} \left(|\nabla v|^2-|\nabla u|^2\right)\, dx- \int_{B_{\sqrt \kappa \varrho}\setminus B_{ \kappa \varrho}} (v-u)\, dx\\
&\le   \int_{(B_{\sqrt \kappa \varrho}\setminus B_{ \kappa \varrho})\cap \{u>w\}} \big(|\nabla w|^2-\nabla u \cdot \nabla w\big)\, dx -\int_{(B_{\sqrt \kappa \varrho}\setminus B_{\kappa \varrho})\cap \{u>w\}} (w-u)\, dx.
\end{split}
\end{equation}
Multiplying the equation satisfied by \(w\) by \((u-w)_+\), integrating over \(B_{\sqrt \kappa\varrho}\setminus B_{\kappa \varrho}\) we obtain
\begin{equation}
\label{fredderrimo}
 \int_{(B_{\sqrt \kappa \varrho}\setminus B_{ \kappa \varrho})\cap \{u>w\}} \big(|\nabla w|^2-\nabla u \cdot \nabla w\big)\, dx -\int_{(B_{\sqrt \kappa \varrho}\setminus B_{ \kappa \varrho})\cap \{u>w\}}(u-w)\, dx=\int_{\partial B_{\kappa\varrho}} \frac{\partial w}{\partial\nu}\, u\, d\mathcal{H}^{N-1},
 \end{equation}
since $w\equiv 0$ on $\partial B_{\kappa\varrho}$ and $w\ge u$ on $\partial B_{\sqrt{\kappa}\varrho}$.
An explicit computation gives
\[
\left|\frac {\partial w}{\partial \nu}\right|
\le C(N,\kappa)\, \frac{\delta_{\varrho}+\varrho^2}{\varrho} \qquad\text{on \(\partial B_{\kappa \varrho}\),}
\]
and combining \eqref{freddissimo} and \eqref{fredderrimo} we get
\begin{equation}
\label{freddo}
\frac 1 2  \int_{B_{ \kappa \varrho}} |\nabla u|^2\, dx-\int_{B_{\kappa\varrho}} u\, dx+ \frac{\widehat \eta}{2}\,\left|\{u>0\}\cap B_{ \kappa \varrho}\right| \le C\,\frac{\delta_{\varrho}+\varrho^2}{\varrho}\,\int_{\partial B_{\kappa \varrho}} u\, d\mathcal{H}^{N-1}.
\end{equation}
Now the classical trace inequality in \(W^{1,1}\), \eqref{sup} and Cauchy-Schwarz inequality imply
\[
\begin{split}
\int_{\partial B_{\kappa \varrho}} u\, d\mathcal{H}^{N-1}&\le C(N,\kappa)\,\left(\frac{1}{\varrho}\int_{B_{\kappa \varrho}}u\, dx+\int_{B_{\kappa \varrho}}|\nabla u|\, dx\right)\\
&\le C(N,\kappa)\Bigg(\Big(\frac{\delta_\varrho}{\varrho}+\frac{1}{2} \Big)\,\big|\{u>0\}\cap B_{\kappa \varrho}\big|+\frac{1}{2} \int_{B_{\kappa \varrho}}|\nabla u|^2\, dx\Bigg).
\end{split}
\]
Combining the above estimate with \eqref{freddo}, recalling \eqref{sup} and choosing \(m\) and \(\varrho_0\) such that $(m+\varrho)\, C(N,\kappa)\le \widehat\eta/4$, we obtain
\[
\begin{split}
\frac{\widehat\eta}{2}\,\left(\int_{B_{ \kappa \varrho}} |\nabla u|^2\, dx + |\{u>0\}\cap B_{ \kappa \varrho}|\right)&\le (m+\varrho)\,C\, \left(\int_{B_{ \kappa \varrho}} |\nabla u|^2\, dx +|\{u>0\}\cap B_{ \kappa \varrho}|\right)\\
&\le\frac{\widehat\eta}{4} \left(\int_{B_{ \kappa \varrho}} |\nabla u|^2\, dx +|\{u>0\}\cap B_{ \kappa \varrho}|\right).
\end{split}
\]
This clearly implies \(u=0\) on \(B_{\kappa \varrho}\).
\end{proof}
\begin{lemma}\label{lemma2}Let \(u_j\) be as in Lemma \ref{lemma1}. There exists a constant \(M\) depending only on the dimension and on \(R\), such that  if \(x_0\in B_R\) and 
\begin{equation}\label{M}
\mean{\partial B_{\varrho}(x_0)\cap B_R}u_j\, d\mathcal{H}^{N-1}\ge M \varrho,
\end{equation}
then \(u>0\) in \(B_{\varrho}(x_0)\).
\end{lemma}

\begin{proof}Again we drop the subscript \(j\). First notice that if \(M\) is large enough and \eqref{M} holds true then necessarily 
\(B_{\varrho}(x_0)\subset B_R\). Indeed, remember that $-\Delta u\le 1$ in $B_R$, then by the maximum principle
\[
u(x)\le \frac{R^2-|x|^2}{2N} \qquad x\in B_R.
\]
Thus, if \(B_\varrho(x_0)\cap \partial B_R\ne \emptyset\) then
\[
\mean{\partial B_{\varrho}(x_0)\cap B_R}u\, d\mathcal{H}^{N-1} \le C(N,R)\, \varrho,
\]
for some constant depending only on \(R\) and \(N\) and this would contradict \eqref{M} if $M>C(N,R)$. Then we can always assume that $B_\varrho(x_0)\subset B_R$. Let us now define \(v\in W^{1,2}_0(B_R)\) as 
\[
\begin{cases}
-\Delta v=1\quad&\text{on \(B_\varrho\)}\\
v=u &\text{in \(\R^N\setminus B_{\varrho}\)},
\end{cases}
\]
where we simply write \(B_\varrho\) for \(B_\varrho(x_0)\), since \(x_0\) is fixed. 
Of course, by the maximum principle there holds \(v>0\) in \(B_\varrho\) and since $u=v$ in the complementary of $B_\varrho$, we get 
\[
\{u>0\}\Delta \{v>0\}=\{u=0\}\cap B_\varrho.
\]
Using this, \eqref{quasimin} and \eqref{f-prop} we get
\[
\frac 1 2 \int_{B_\varrho} |\nabla u|^2\, dx-\int_{B_\varrho} u\, dx\le \frac 1 2  \int_{B_\varrho} |\nabla v|^2\, dx-\int_{B_\varrho} v\, dx+\left(\frac {1}{\widehat \eta}+C_2\,\sigma\right) |\{u=0\}\cap B_\varrho|.
\]
By appealing to the equation satisfied by \(v\) and the fact \(\sigma < 1<1/\widehat \eta\), the above equation becomes
\begin{equation}
\label{uv}
\frac 1 2 \int_{B_\varrho}|\nabla u -\nabla v|^2\, dx\le \frac{C_2+1}{\widehat \eta}\, \big|\{u=0\}\cap B_\varrho\big|.
\end{equation}
Through the scaling
\[
u(x)\mapsto \frac{1}{\varrho}\, u(x_0+\varrho\, x),
\]
we can assume that \(\varrho=1\). We want to bound the left-hand side of \eqref{uv} from below by a multiple of the right-hand side. In order to do this  we fix  two points \(y_1\) and \(y_2\) in \(B_{1/4}\) such that \(B_{1/8}(y_1)\) and \(B_{1/8}(y_2)\) are disjoint and contained in \(B_{1/2}\). For \(i=1,2\), let \(\zeta_i:\mathbb{S}^{N-1}\to \mathbb{R}^+\) be such that
\begin{equation}
\label{zeta}
\partial B_1=\{y_i+\zeta_i(\theta)\,\theta\, :\, \theta\in\mathbb{S}^{N-1}\}.
\end{equation}
Let us define
\[
\psi_i(\theta)=y_i+r_i(\theta)\,\theta
\]
where
\[
r_i(\theta)=\inf\left\{ \frac 1 8 \le r\le \zeta_i(\theta)\, :\, u(y_i+r\theta)=0\right\}.
\]
and we set the above infimum to be \(\zeta_i(\theta)\) if no such \(r\) exists. That is \(\psi_i(\theta)\) is the first point  outside \(B_{1/8}\) and lying on the segment joining \(y_i\) to \(y_i+\zeta_i(\theta)\,\omega\) where \(u\) vanishes. Hence
\begin{equation}\label{tfc}
\begin{split}
v(\psi_i(\theta))&=v(\psi_i(\theta))-u(\psi_i(\omega))\\
&\le \int_{r_i(\theta)}^{\zeta_i(\theta)}|\nabla u-\nabla v|(y_i+r\omega) \,dr\\
&\le\sqrt{ \zeta_i(\theta)-r_i(\theta)}\,\left( \int_{r_i(\theta)}^{\zeta_i(\theta)}|\nabla u-\nabla v|^2(y_i+r\omega)| \,dr\right)^{1/2}. 
\end{split} 
\end{equation}
By the maximum principle \(v\) is above the harmonic function sharing the same boundary data of \(u\), hence, by the Poisson representation formula it follows that
\begin{equation}\label{poisson}
\begin{split}
v(\psi_i(\theta))&\ge c(N) \big(1-|\psi_i(\theta)|\big)\, \mean{\partial B_1} u\, d\mathcal{H}^{N-1}.
\end{split}
\end{equation}
\begin{figure}
\includegraphics[scale=.35]{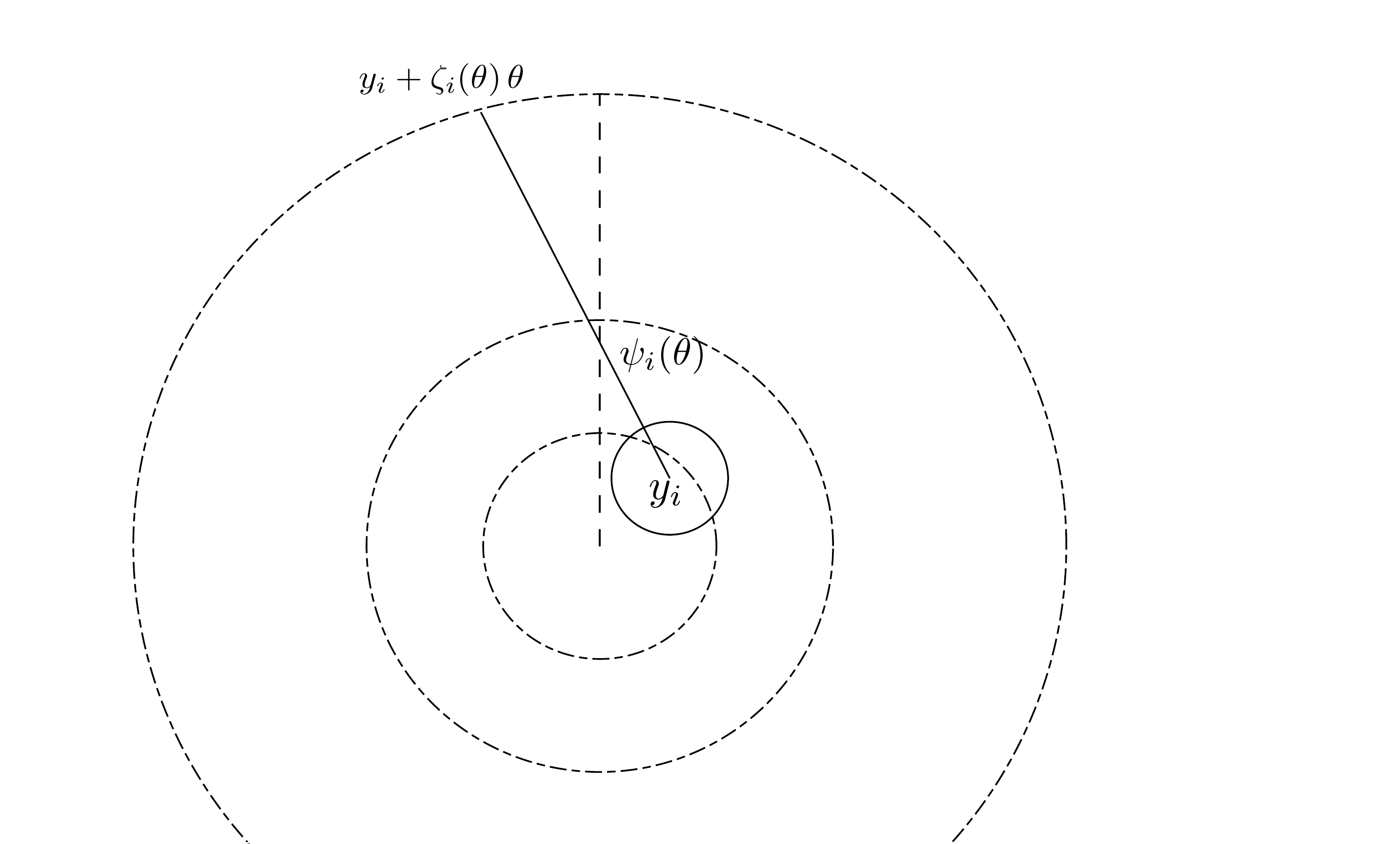}
\caption{The geometric construction of Lemma \ref{lemma2}}
\end{figure}
By elementary geometric considerations (see Figure 4.1),
\[
1-|\psi_i(\theta)|\ge c\, |\zeta_i(\theta)-r_i(\theta)|,
\]
and by construction
\[
|\zeta_i(\theta)-r_i(\theta)|\ge \mathcal H^{1}\big(\{r\in [1/8,\zeta_i(\theta)]\, :\,  u(y_i+r\theta)=0\}\big).
\]
Then integrating over \(\theta\in \mathbb S^{N-1}\) and using that $r_i(\theta)\ge 1/8$, \eqref{tfc},  \eqref{poisson}, \eqref{uv} and our assumption (recall that we have set  \(\varrho=1\)) imply
\[
M^2\, \big|\{u=0\}\cap (B_1\setminus B_{1/8}(y_i))\big|\le C(N,R)\,\big|\{u=0\}\cap B_1\big|.
\]
Since the balls \(B_{1/8}(y_1)\) and \(B_{1/8}(y_2)\) are disjoint, this gives
\[
\frac{M^2}{2}|\{u=0\}\cap B_1|\le C(N,R)\, |\{u=0\}\cap B_1|.
\]
By choosing $M$ large enough, the previous implies that $|\{u=0\}\cap B_1|=0$.
\end{proof}
From Lemma \ref{lemma1} and \ref{lemma2}, exactly as in \cite[Section 3]{AC}, we obtain the following. 
\begin{lemma}[Properties of  minimizers, Part II]\label{lm:prop2} Let \(u_j\) be as above, then \(\Omega_j=\{u_j>0\}\) is an open set. Moreover there exists a constant \(C_6=C_6(N,R)\) and a radius \(\varrho_0=\varrho_0(N,R)\) such that
\begin{itemize}
\item[(i)] For every \(x\in \Omega_j\) it holds 
\begin{equation}\label{nondegen}
\frac{1}{C_6}\, \dist(x,\partial \Omega_j)\le u_j(x)\le C_6\,\dist(x,\partial\Omega_j);
\end{equation}
\item[(ii)] the functions \(u_j\) are equi-Lipschitz \(\|\nabla u_j\|_{L^\infty (B_R)}\le C_6\);
\item[(iii)] for every \(x\in \partial\Omega_j\) and every \(\varrho\le \varrho_0\)
\begin{equation}\label{density}
\frac{1}{C_6}\le \frac{|\Omega_j\cap B_\varrho (x)|}{|B_\varrho(x)|}\le C_6.
\end{equation}
\end{itemize}
\end{lemma}
As in \cite[Theorem 4.5]{AC} we also have the following.
\begin{lemma}\label{weak}Let \(u_j\) as above, then  there exists a Borel function \(q_{u_j}\) such that
\begin{equation}\label{weaksol}
-\Delta u_j=1_{\Omega_j}-q_{u_j}\,\mathcal H^{N-1}\llcorner \partial^* \Omega_j.
\end{equation}
In addition \(0<c\le q_{u_j} \le C\), where \(c\) and \(C\) depends only on \(N\) and \(R\) and 
\[
\mathcal H^{N-1}\big( \partial \Omega_j\setminus \partial^*\Omega_j\big)=0.
\]
\end{lemma}

In the above Lemma \(\partial^* \Omega_j\) denotes the \emph{reduced boundary} of the set of finite perimeter \(\Omega_j=\{u_j>0\}\). We recall (see \cite[Chapter 15]{Maggi}) that for every \(\bar x\in \partial^* \{u_j>0\}\), there exists a unit vector \(\nu_{u_j}(\bar x)\) such that
\begin{equation}\label{halfplane}
\frac{\{u_j>0\}-\bar x}{\varrho} \rightarrow \big\{x:\ x\cdot\nu_{u_j}(\bar x)\ge 0\big\},\qquad \text{in } L^1_{\rm loc} (\R^N).
\end{equation}
Moreover for \(\mathcal H^{N-1}\) almost every\footnote{More precisely, in every Lebesgue point of \(q_{u_j}\) with respect to \(\mathcal H^{N-1}\llcorner \partial^*\{u_j>0\}\).}  \(\bar x \in \partial^* \{u_j>0\}\), it holds
\begin{equation}
\label{blowup}
u_j^{\varrho}(x):=\frac{u_j(\bar x+\varrho x)}{\varrho} \to q_{u_j}(\bar x) \big(x\cdot \nu_{u_j}(\bar x)\big)_+,\qquad \text{ in \( W^{1,p}_{\rm loc}(\R^N)\) for every \(p\in[1,+\infty).\)}
\end{equation}
For the proofs of this last  fact we refer to \cite[Theorem 4.8]{AC}. 
The following simple Lemma is a standard consequence of Lemma \ref{lm:prop1} (ii) and of the density estimates \eqref{density}.
\begin{lemma}
\label{kuraconv}
Every limit point \(\Omega_\infty\) of \(\Omega_j\) with respect to the \(L^1\) convergence is a ball of radius \(1\) and center \(x_\infty \in B_R\). Moreover 
\[
\partial \Omega_j \to \partial B_1(x_\infty)\qquad\text{in the Kuratowski sense}\footnote{We recall that a sequence of sets \(S_k\) converges to a set \(S\) in the Kuratowski sense if \begin{itemize} \item for every sequence of points \(x_k\in S_k\) any limit point belongs to  \(S\);\item for every point \(x\in S\) there is a sequence of point \(x_k\in S_k\) such that \(x_k\to x\).\end{itemize}} \text{ as \(j\to \infty\)}.
\]
In particular for every \(\delta>0\) there exists a \(j_\delta\in\mathbb{N}\) such that
\begin{equation}\label{kura}
B_{1-\delta}(x_j)\subset \Omega_j \subset B_{1+\delta}(x_j),\qquad \mbox{ for all }\, j\ge j_{\delta},
\end{equation}
where \(x_j\in B_R\).
\end{lemma}
We are now in position to address higher regularity of \(\partial\{u_j>0\}\). Since  \(u_j\) is a weak solution for \(q_{u_j}\) in the sense of \cite[Section 5]{AC} and \cite[Section 3]{GS}. To apply their results we have to show that \(q_{u_j}\) is continuous. To identify \(q_{u_j}\) we try to write down the Euler-Lagrange equations for the problem \eqref{umin}. In order to do this we first  have to show that \(\Omega_j=\{u_j>0\}\) is minimal with respect to every (small) perturbation. This will be done in the next Lemma, where by \(\mathcal N_{\delta}(A)\) we denote  the \(\delta\) neighborhood of a generic set \(A\).
\begin{lemma}
\label{smallvariation}
Let \(R\ge2\), then for every \(0<2\,\delta<R-1\)  there exists \(j_\delta\in\mathbb{N}\) such that for \(j\ge j_\delta\), the energy function \(u_j\) satisfies \eqref{umin} for every \(v\in W^{1,2}_0\big(\mathcal N_{\delta}(\Omega_j)\big)\). 
\end{lemma}
\begin{proof} Let \(\delta\) be as in the statement, thanks to \eqref{kura} we can assume that for \(j\ge j_\delta\) we have \(\Omega_j\subset B_{1+\delta}(x_j)\) for some \(x_j\in B_R\). The translated sets $U_j=\Omega_j-x_j$ are such that
\[
U_j \subset B_{1+\delta}(0)\subset B_R,
\]
so that $\mathcal{N}_\delta(U_j)\subset B_R$. Moreover, we have \(\G_{\widehat\eta,j}(\Omega_j)=\G_{\widehat\eta,j}(U_j)\), as the functional is translation invariant. 
Translating back this proves the claim.  
\end{proof}
It is not difficult to see that the formal optimality condition for \eqref{umin} reads as 
\[
\Bigg|\frac{\partial u_j(x)}{\partial \nu}\Bigg|^2-\frac{\sigma^2 (\alpha(\Omega_j)-\eps_j)}{\sqrt{\eps_j^2+\sigma^2(\alpha(\Omega_j)-\eps_j)^2}}\Bigg[|x-x_{\Omega_j}|-\Bigg(\int_{\Omega_j}\frac{y-x_{\Omega_j}}
{|y-x_{\Omega_j}|}dy\Bigg)\cdot x\Bigg]=\Lambda_j,
\]
for some constant \(\Lambda_j\). The goal of next Lemma is to show that this is actually the case, at least for \(\mathcal H^{N-1}\llcorner \partial^*\Omega_j\) almost every point.
\begin{lemma}\label{lm:EL}
Let \(R\ge 2\) and \(u_j\) be as in Lemma \ref{lm:prop1}. There exists $\overline j=\overline j(R)$ such that if $j\ge \overline j$, the following holds:
\begin{equation}
\label{EL}
\begin{split}
\big(q_{u_j}(x_1)\big)^2&-\frac{\sigma^2 (\alpha(\Omega_j)-\eps_j)}{\sqrt{\eps_j^2+\sigma^2(\alpha(\Omega_j)-\eps_j)^2}}\Bigg[|x_1-x_{\Omega_j}|-\Bigg(\mean{\Omega_j}\frac{y-x_{\Omega_j}}
{|y-x_{\Omega_j}|}dy\Bigg)\cdot x_1\Bigg]\\
&=
\big(q_{u_j}(x_2)\big)^2-\frac{\sigma^2 (\alpha(\Omega_j)-\eps_j)}{\sqrt{\eps_j^2+\sigma^2(\alpha(\Omega_j)-\eps_j)^2}}\Bigg[|x_2-x_{\Omega_j}|-\Bigg(\mean{\Omega_j}\frac{y-x_{\Omega_j}}
{|y-x_{\Omega_j}|}dy\Bigg)\cdot x_2\Bigg],
\end{split}
\end{equation}
for every two points \(x_1\) and \(x_2\) in \(\partial^*\{u_j>0\}\) such that \eqref{halfplane} and \eqref{blowup} hold true.
\end{lemma}

\begin{proof}
We choose \(\delta=(R-1)/4\) and fix \(j\ge j_\delta\), where \(j_\delta\) is as in Lemma \ref{smallvariation}. Being \(j\) fixed we drop the subscript and, for notational simplicity, we assume that \(x_\Omega=0\).
\par
Let us assume by contradiction that there are two points \(x_1\) and \(x_2\) satisfying \eqref{halfplane} and \eqref{blowup} such that the left-hand side of \eqref{EL} is strictly smaller than the right-hand side. 
\par
Following \cite{AAC} we are going to construct a small variation of \(\Omega=\{u>0\}\) which preserves the volume to the first order and which contradicts \eqref{umin}. In order to do this let us take a smooth radial symmetric function \(\phi(y)=\phi(|y|)\) compactly supported in \(B_1\) and let us define, for \(\tau,\varrho\) small
\begin{equation}
\Phi^\varrho_\tau(x)=
\left\{\begin{array}{cc}
x+\tau\varrho\,\phi\Bigg(\Big|\dfrac{x-x_1}{\varrho}\Big|\Bigg)\nu_{u}(x_1), &x\in B_{\varrho}(x_1),\\
\\
x-\tau\varrho\,\phi\Bigg(\Big|\dfrac{x-x_2}{\varrho}\Big|\Bigg)\nu_{u}(x_2), &x\in B_{\varrho}(x_2),\\
\\
x,&\text{otherwise.}
\end{array}
\right.
\end{equation}
For \(\tau\) small and independent of \(\varrho\), \(\Phi_\tau^\varrho\) is easily seen to be a diffeomorphism. Moreover, still for $\tau$ small, thanks to Lemma \ref{smallvariation} the function
\[
u_\tau^\varrho:=u\circ \Big(\Phi_\tau^\varrho\Big)^{-1}, 
\]
is an admissible competitor for testing the minimality of \(u\),
notice that 
\[
\Omega_\tau^\varrho:=\{x\, :\, u_\tau^\varrho(x)>0\}=\Phi^\varrho_\tau(\Omega).
\]
We now start computing the variations of all the terms involved in the definition of $\mathcal{G}_{\widehat\eta,j}$.
\vskip.2cm\noindent
$\bullet$ {\it Volume term}. We compute
\begin{equation*}
\begin{split}
\frac{|\Omega_\tau^\varrho|-|\Omega|}{\varrho^N}&=\frac{1}{\varrho^N} \int_{\Omega} [\det\nabla \Phi^\varrho_\tau-1]\, dx\\
&=\frac{\tau}{\varrho^N} \Bigg( \int_{\Omega\cap B_{\varrho}(x_1)} \phi'\Bigg(\Big|\dfrac{x-x_1}{\varrho}\Big|\Bigg)\frac{(x-x_1)\cdot\nu_u (x_1)}{|x-x_1|}\, dx\\
&-\int_{\Omega\cap B_{\varrho}(x_2)} \phi'\Bigg(\Big|\dfrac{x-x_2}{\varrho}\Big|\Bigg)\frac{(x-x_2)\cdot\nu_u (x_2)}{|x-x_2|}\, dx\Bigg)+o(\tau),
\end{split}
\end{equation*}
where \(o(\tau)\) is independent on \(\varrho\). Hence, recalling \eqref{halfplane}, changing variables and applying the Divergence Theorem in the last step
\begin{equation}\label{volume}
\begin{split}
\lim_{\varrho\to 0}\frac{|\Omega_\tau^\varrho|-|\Omega|}{\varrho^N} &=\tau\Bigg(\int_{\{y\cdot \nu_u(x_1)\ge 0\}\cap B_1}\phi'(|y|)\,\frac{y\cdot \nu_u(x_1)}{|y|}-\int_{\{y\cdot \nu_u(x_2)\ge 0\}\cap B_1}\phi'(|y|)\frac{y\cdot \nu_u(x_2)}{|y|}\Bigg)+o(\tau)\\
&=-\tau \left(\int_{\{y\cdot\nu_u(x_1)=0\}\cap B_1} \phi(|y|)-\int_{\{y\cdot \nu_u(x_2)=0\}\cap B_1} \phi(|y|)\right)+o(\tau)=o(\tau),
\end{split}
\end{equation}
where we used that the integrals are equal due to the radial symmetry of \(\phi\).  
\vskip.2cm\noindent
$\bullet$ {\it Dirichlet integral and $L^1$ norm}. By changing variables,
\begin{equation*}
\begin{split}
&\frac{1}{\varrho^N}\Bigg(\int |\nabla u ^\varrho_\tau|^2-\int |\nabla u|^2 \Bigg)\\
&=\frac{1}{\varrho^N}\Bigg(\int_{\Omega\cap B_{\varrho}(x_1)} \Big[|(\nabla \Phi_\tau^\varrho)^{-1}\nabla u|^2\det\nabla \Phi_\tau^\varrho-|\nabla u|^2\Big] +\int_{\Omega\cap B_{\varrho}(x_2)} \Big[|(\nabla \Phi_\tau^\varrho)^{-1}\nabla u|^2\det\nabla \Phi_\tau^\varrho-|\nabla u|^2\Big]\Bigg)\\
&=\frac{\tau }{\varrho^N}\Bigg(\int_{\frac{\Omega-x_1}{\varrho}\cap B_1}|\nabla u(x_1+\varrho y)|^2\phi'(|y|)\frac{ \nu_u (x_1) \cdot y}{|y|}-2\phi'(|y|)\frac{ (\nabla u(x_1+\varrho y)\cdot y)(\nabla u(x_1+\varrho y)\cdot\nu_u(x_1))}{|y|}\\
&-\int_{\frac{\Omega-x_2}{\varrho}\cap B_1}|\nabla u(x_2+\varrho y)|^2\phi'(|y|)\frac{ \nu_u (x_2) \cdot y}{|y|}-2\phi'(|y|)\frac{ (\nabla u(x_2+\varrho y)\cdot y)(\nabla u(x_2+\varrho y)\cdot\nu_u(x_2))}{|y|} \Bigg)+o(\tau),
\end{split}
\end{equation*}
with \(o(\tau)\) independent on \(\varrho\). Hence, recalling \eqref{blowup} and \eqref{halfplane}, thanks to the Divergence Theorem we obtain
\begin{equation}\label{dir}
\begin{split}
\lim_{\varrho\to 0} \frac{1}{\varrho^N}\Bigg(\int |\nabla u ^\varrho_\tau|^2-\int |\nabla u|^2 \Bigg)
=\tau \Big[\big(q_u(x_1)\big)^2-\big(q_u(x_2)\big)^2\Big]\int_{\{y_1=0\}\cap B_1} \phi(|y|)+o(\tau).
\end{split}
\end{equation}
With a similar computation and recalling~\eqref{blowup},
\begin{equation}\label{u}
\lim_{\varrho\to 0}\frac{1}{\varrho^N}\Bigg(\int u ^\varrho_\tau-\int u \Bigg)=o(\tau).
\end{equation}
\vskip.2cm\noindent
$\bullet$ {\it Barycenter}. First of all, recall that we have set \(x_\Omega=0\). So we only have to compute
\begin{equation*}
\begin{split}
\frac{x_{\Omega_\tau^\varrho}}{\varrho^N}&=\frac{1}{\varrho^N}\Bigg(\frac{1}{|\Omega_\tau^\varrho|}\int_{\Omega}\Phi_\tau^\varrho(x)\det\nabla\Phi_\tau^\varrho(x)-\frac{1}{|\Omega|}\int_\Omega x\Bigg)\\
&=\frac{1}{\varrho^N}\Bigg(\frac{1}{|\Omega|}\int_{\Omega}\Phi_\tau^\varrho(x)\det\nabla\Phi_\tau^\varrho(x)-\frac{1}{|\Omega|}\int_\Omega x\Bigg)+o_{\varrho}(1)+o(\tau),
\end{split}
\end{equation*}
where we have taken into account \eqref{volume} in the second equality and \(o_\varrho(1)\) tends to \(0\) in \(\varrho\) for fixed \(\tau\), while \(o(\tau)\) is independent on \(\varrho\). Arguing as above, one checks that thanks to \eqref{halfplane},
\begin{equation}
\label{bar}
\begin{split}
\lim_{\varrho\to 0}\frac{x_{\Omega_\tau^\varrho}}{\varrho^N}&=\lim_{\varrho\to 0}\frac{1}{\varrho^N}\left(\frac{1}{|\Omega|}\int_{\Omega}\Phi_\tau^\varrho(x)\det\nabla\Phi_\tau^\varrho(x)-\frac{1}{|\Omega|}\int_\Omega x\right)+o(\tau)\\
&=-\tau\,\frac{(x_1-x_2)}{|\Omega|}\Bigg(\int_{\{y_1=0\}\cap B_1}\phi(|y|)\Bigg)+o(\tau).
\end{split}
\end{equation}
$\bullet$ {\it Asymmetry}. Recalling \eqref{alphacomodo} and that we have set \(x_\Omega=0\),
\begin{equation*}
\begin{split}
\frac{\alpha(\Omega_\tau^\varrho)-\alpha(\Omega)}{\varrho^N}
&=\frac{1}{\varrho^N}\Bigg( \int_{\Omega\cap B_{\varrho}(x_1)} \Big[|\Phi_\tau^\varrho(x)|\det\nabla\Phi_\tau^\varrho(x)-|x|\Big]\\
&+\int_{\Omega\cap B_{\varrho}(x_2)}\Big[|\Phi_\tau^\varrho(x)|\det\nabla\Phi_\tau^\varrho(x)-|x|\Big]-\Bigg[\int_\Omega \frac{y}{|y|}\Bigg]\cdot x_{\Omega_\tau^\varrho} \Bigg)\\
&+o_\varrho(1)+o(\tau)
\end{split}
\end{equation*}
where we used \eqref{volume} and \eqref{bar}. Here again \(o_\varrho(1)\) tends to \(0\) in \(\varrho\) for fixed \(\tau\), while \(o(\tau)\) is independent on \(\varrho\).
With a computation similar to the previous ones
\begin{equation}\label{dist}
\lim_{\varrho\to 0}\frac{1}{\varrho^N}\Bigg( \int_{\Omega\cap B_{\varrho}(x_i)}|\Phi_\tau^\varrho(x)|\det\nabla\Phi_\tau^\varrho(x)-|x|\Bigg)=\tau(-1)^{i}\,|x_i|\,\int_{\{y_1=0\}\cap B_1} \phi(|y|)+o(\tau).
\end{equation}
for \(i=1,2\). Hence we finally get
\begin{equation}
\label{al1}
\lim_{\varrho\to 0} \frac{\alpha(\Omega_\tau^\varrho)-\alpha(\Omega)}{\varrho^N}=-\tau\,\left(\int_{\{y_1=0\}\cap B_1} \phi(|y|)\right)\, \left(|x_1|-|x_2|+\left[\int_\Omega \frac{y}{|y|} \right]\cdot\frac{x_1-x_2}{|\Omega|} \right)+o(\tau).
\end{equation}
\vskip.2cm\noindent
$\bullet$ {\it Expansion of $\mathcal{G}_{\widehat\eta,j}$}. By \eqref{volume} and \eqref{f-prop}, we get
\[
\frac{1}{\varrho^N}\Big|f_{\widehat \eta}(|\Omega_\tau^\varrho|)-f_{\widehat \eta}(|\Omega|)\Big|\le \frac{1}{\widehat\eta}\Bigg|\frac{|\Omega_\tau^\varrho|-|\Omega|}{\varrho^N}\Bigg|=o_\varrho(1)+o(\tau).
\]
Thus by using \eqref{dir}, \eqref{u}, \eqref{bar} and \eqref{al1} we can infer
\[
\begin{split}
\left(\int_{\{y_1=0\}\cap B_1} \phi(|y|)\right)^{-1}&\frac{\G_{\widehat\eta,j}(\Omega_\tau^\varrho)-\G_{\widehat\eta,j}(\Omega)}{\varrho^N}\\
&\le\tau \Bigg(\big(q_{u}(x_1)\big)^2-\frac{\sigma^2 (\alpha(\Omega)-\eps)}{\sqrt{\eps^2+\sigma^2(\alpha(\Omega)-\eps)^2}}\Big[|x_1|-\Big(\mean{\Omega}\frac{y}
{|y|}dy\Big)\cdot x_1\Big]\\
&-\big(q_{u}(x_2)\big)^2+\frac{\sigma^2 (\alpha(\Omega)-\eps)}{\sqrt{\eps^2+\sigma^2(\alpha(\Omega)-\eps)^2}}\Big[|x_2|-\Big(\mean{\Omega}\frac{y}
{|y|}dy\Big)\cdot x_2\Big]\Bigg)+o_\varrho(1)+o(\tau),
\end{split}
\]
which contradicts the minimality of \(\Omega\) for \(\varrho\), \(\tau\) small.
\end{proof}
\begin{lemma}
\label{lm:qu}
There exist $\sigma_3=\sigma_3(N,R)>0$, $\widehat{j}=\widehat{j}(N,R)$ and $\widehat\delta=\widehat\delta(N,R)>0$ such that for every $j\ge \widehat{j}$ and every $\sigma\le \sigma_3$ the functions $q_{u_j}$ are in $C^\infty(\mathcal{N}_{\widehat\delta}(\partial\Omega_j))$. Moreover
\[
\|q_{u_j}\|_{C^k(\mathcal{N}_{\widehat\delta}(\partial\Omega_j))}\le C(k,N,R),\qquad \mbox{ for every }j\ge \widehat j.
\]  
\end{lemma}
\begin{proof}
From Lemma \ref{lm:EL} we see that, for \(j\) large there exists $\Lambda_{j}\in\mathbb{R}$ such that
\begin{equation}
\label{dernor}
q_{u_j}(x)^2=\Lambda_j+\frac{\sigma^2 (\alpha(\Omega_j)-\eps_j)}{\sqrt{\eps_j^2+\sigma^2(\alpha(\Omega_j)-\eps_j)^2}}\Bigg[|x-x_{\Omega_j}|-\Bigg(\mean{\Omega_j}\frac{y-x_{\Omega_j}}
{|y-x_{\Omega_j}|}dy\Bigg)\cdot x\Bigg],
\end{equation}
for \(\mathcal H^{N-1}\)-almost every \(x\in\partial\{u_j>0\}\). Since 
\[
\left|\frac{\sigma^2 (\alpha(\Omega_j)-\eps_j)}{\sqrt{\eps_j^2+\sigma^2(\alpha(\Omega_j)-\eps_j)^2}}\Bigg[|x-x_{\Omega_j}|-\Bigg(\mean{\Omega_j}\frac{y-x_{\Omega_j}}
{|y-x_{\Omega_j}|}dy\Bigg)\cdot x\Bigg]\right|\le C(N,R)\,\sigma,
\]
and by Lemma \ref{weak} \(q_{u_j}\) is bounded from above and below independently on \(j\), there exists a $\sigma_3=\sigma_3(N,R)$ such that for \(\sigma\le \sigma_3\) we have
 that \(\Lambda_j\) is also bounded from above and below independently on \(j\). Thanks to \eqref{kura} 
\[
|x-x_{\Omega_j}|\ge \frac{1}{2},\qquad \mbox{ for every } j \mbox{ large}.
\]
Hence we can find $\widehat\delta=\widehat\delta(N,R)$ such that
\[
q_{u_j}(x)=\left(\Lambda_j+\frac{\sigma^2 (\alpha(\Omega_j)-\eps_j)}{\sqrt{\eps_j^2+\sigma^2(\alpha(\Omega_j)-\eps_j)^2}}\Bigg[|x-x_{\Omega_j}|-\Bigg(\mean{\Omega_j}\frac{y-x_{\Omega_j}}
{|y-x_{\Omega_j}|}dy\Bigg)\cdot x\Bigg]\right)^{1/2}
\]
is smooth in the neighborhood $\mathcal{N}_{\widehat\delta}(\partial\Omega_j)$ and all its $C^k$ norms are bounded, independently of $j$.
\end{proof}
We are in the position to apply the results\footnote{See also \cite[Appendix]{GS}, where it is sketched how to modify the proofs in \cite{AC} to deal with the case in which the function \(u\) has bounded laplacian on \(\{u>0\}\).} of Sections 7 and 8 of \cite{AC}. We start recalling the following definition, see \cite[Definition 7.1]{AC}.

\begin{definition}Let \(\mu_-,\mu_+\in (0,1]\), \(\kappa>0\). A weak solution \(u\) of \eqref{weaksol} is said to be of class \(F(\mu_-, \mu_+,\kappa)\) in  \(B_{\varrho} (x_0)\) with respect to a direction \(\nu\in \mathbb{S}^{N-1}\) if (see Figure 4.2)

\begin{enumerate}
\item[(a)] \(x_0\in \partial \{u>0\}\) and 
\begin{align*}
&u=0 &\qquad&\text{for \((x-x_0)\cdot \nu\le -\mu_-\varrho\),}\\
&u(x)\ge q_u(x_0)\big[(x-x_0)\cdot \nu-\mu_+\varrho\big] &&\text{for \((x-x_0)\cdot \nu\ge \mu_+\varrho\).}
\end{align*} 
\item[(b)] \(|\nabla u (x_0)|\le q_u(x_0)(1+\kappa)\) in \(B_\varrho(x_0)\) and \(\osc_{B_\varrho(x_0)} q_u \le \kappa\, q_u(x_0)\).
\end{enumerate}
\begin{figure}
\includegraphics[scale=.4]{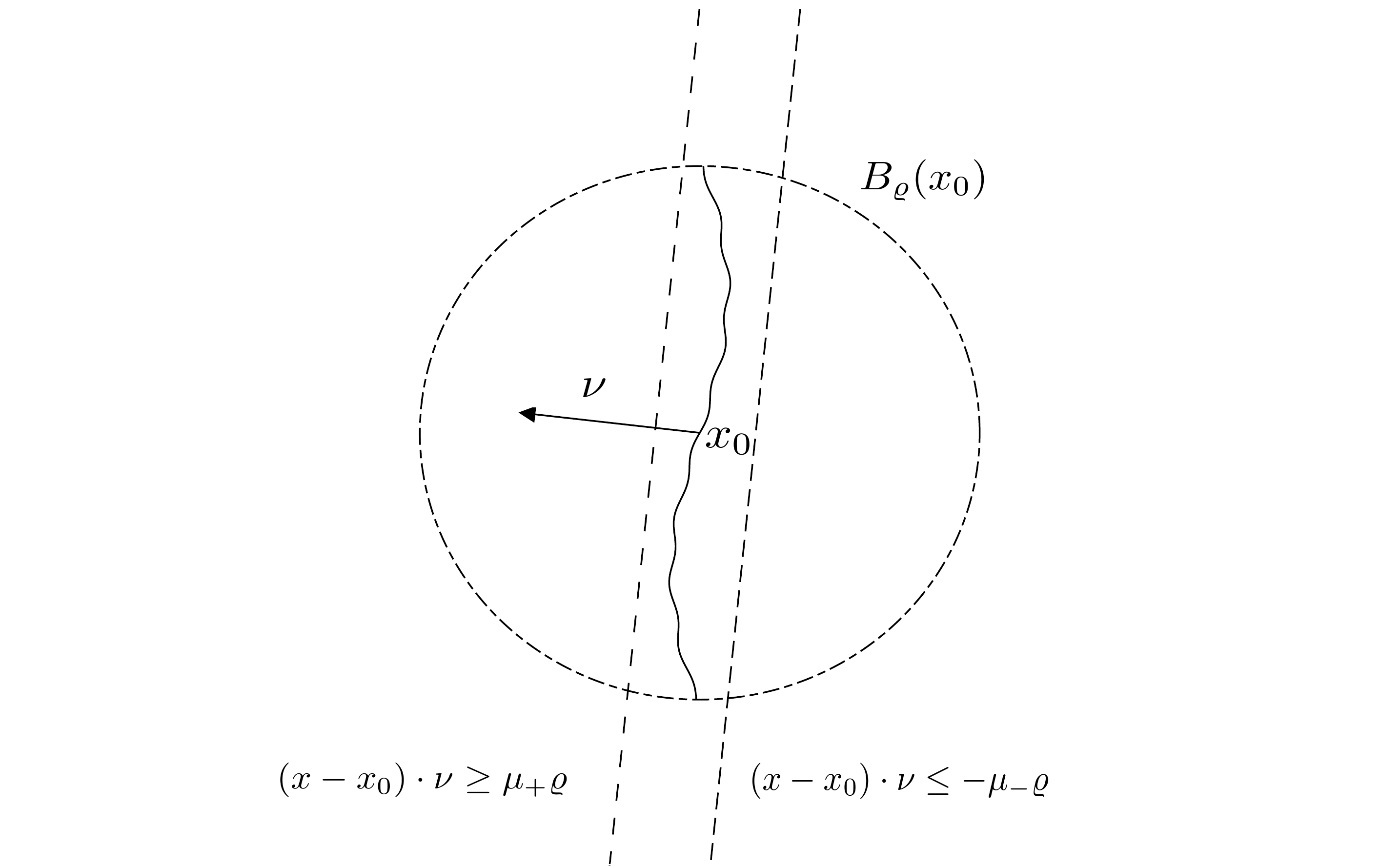}
\label{fig:weak}
\caption{A weak solution of \eqref{weaksol} of class \(F(\mu_-, \mu_+,\kappa)\) in  \(B_{\varrho} (x_0)\) with respect to the direction \(\nu\).}
\end{figure}
\end{definition}
With this definition we can state the main Theorem of \cite{AC}, 
which states that if the free boundary is flat enough, then it is smooth.
\begin{theorem}\cite[Theorem 8.1]{AC}\cite[Theorem 2]{KN}.
\label{thm:ac} Let \(u\) be a weak solution of \eqref{weaksol} in \(B_R\) and assume that \(q_u\) is Lipschitz continuous. There are constants \(\gamma,\bar \mu,\bar \kappa\) and \(C\) depending only on \(\min q_u\), \(\max q_u\), $\mathrm{Lip}(q_u)$, \(R\) and the dimension $N$ such that:
\vskip.1cm\noindent
If \(u\) is of class \(F(\mu,1,+\infty)\) in \(B_{4\varrho}(x_0)\) with respect to some  direction \(\nu\in\mathbb{S}^{N-1}\) with  \(\mu \le \bar \mu\) and  \(\varrho \le \bar \kappa\, \mu^2\), then there exists a \(C^{1,\gamma}\) function \(f: \R^{N-1}\to \R\) with \(\|f\|_{C^{1,\gamma}}\le C \mu\), such that, if we define
 \[
 {\rm graph}_{\nu} f=\big\{x\in\mathbb{R}^N: x\cdot \nu=f(x-(x\cdot \nu)\nu)\big\},
 \]
 then
 \[
 \partial \{u>0\} \cap B_{\varrho}(x_0)=\big(x_0+{\rm graph}_{\nu} f \big)\cap B_{\varrho}(x_0).
 \]
Moreover if \(q_u \in C^{k,\gamma}\) of some $\delta-$neighborhood of $\{u_j>0\}$, then \(f\in C^{k+1,\gamma}\) and   \(\|f\|_{C^{k+1,\gamma}}\le C\,\big(N, R,\|q_u\|_{C^{k,\gamma}}\big)\).
\end{theorem}
\subsection{Proof of the Selection Principle}
With the aid of Theorem \ref{thm:ac}, we can now prove Proposition \ref{thm:sel}.
\begin{proof}[Proof of Proposition \ref{thm:sel}]
Let \(\Omega_j=\{u_j>0\}\) be the solutions of \eqref{mingj} and assume, up to translations, that \(x_{\Omega_j}=0\). Let  $\bar\mu$ be as in Theorem \ref{thm:ac} and let \(\mu\ll \bar \mu \) to be fixed later.
 By the smoothness of \(\partial B_1\), there exists a \(\varrho(\mu)\) such that for every \(\varrho\le \varrho(\mu)\),  \(\bar x\in \partial B_1\) 
\[
\partial B_1\cap B_{5\varrho}( \bar x) \subset \big\{x:\  |(x-\bar x)\cdot \nu_{\bar x}|\le \mu\,\varrho\big\},  
\]
 where \(\nu_{\overline x}\) is the interior normal to \(\partial B_1\) at \(\bar x\). We can also assume that \(\varrho (\mu) \le \bar \tau\, \mu^2\) where \(\bar \tau \) is as in Theorem \ref{thm:ac}. Since, up to translation, by Lemma \ref{kuraconv} \(\partial \Omega_j\) are converging in the sense of Kuratowski  to \(\partial B_1\), for \(j\) large (depending on \(\mu\)) 
 there exists a point \(x_0 \in \partial \Omega_j\cap B_{\mu \varrho (\mu)} (\bar{x})\) such that 
\[
\begin{split}
\partial \Omega_j \cap B_{4\varrho(\mu)}( x_0) \subset \mathcal N_{\mu\varrho(\mu)} \Big( \partial B_1\cap B_{5\varrho(\mu)}( \bar x) \Big) \subset  \big\{x:\  |(x-x_0)\cdot \nu_{\bar x}|\le 4\mu\varrho(\mu)\big\}.  
\end{split}
\] 
This means that \(u_j\) is of class \(F(\mu,1,+\infty)\) in \(B_{4\varrho(\mu)}(x_0)\) with respect to the direction \(\nu_{\bar x}\) and hence, by our assumptions on \(\mu\) and \(\varrho(\mu)\), Lemma \ref{lm:qu} and Theorem \ref{thm:ac}, \(\partial \Omega_j \cap B_{\varrho(\mu)}(x_0)\) is the graph of a smooth function  with respect to \(\nu_{\bar x}\). Choosing \(\mu\) smaller we see that there exist smooth functions \(g^{\bar x}_j\) with uniformly bounded \(C^{k}\) norms such that
\[
\partial \Omega_j \cap B_{\varrho(\mu)}(\bar x)=\Big\{x+g^{\bar x}_j(x)\,x\, :\, x\in \partial B_1\Big\}\cap  B_{\varrho(\mu)}(\bar x).
\] 
\begin{figure}
\includegraphics[scale=.45]{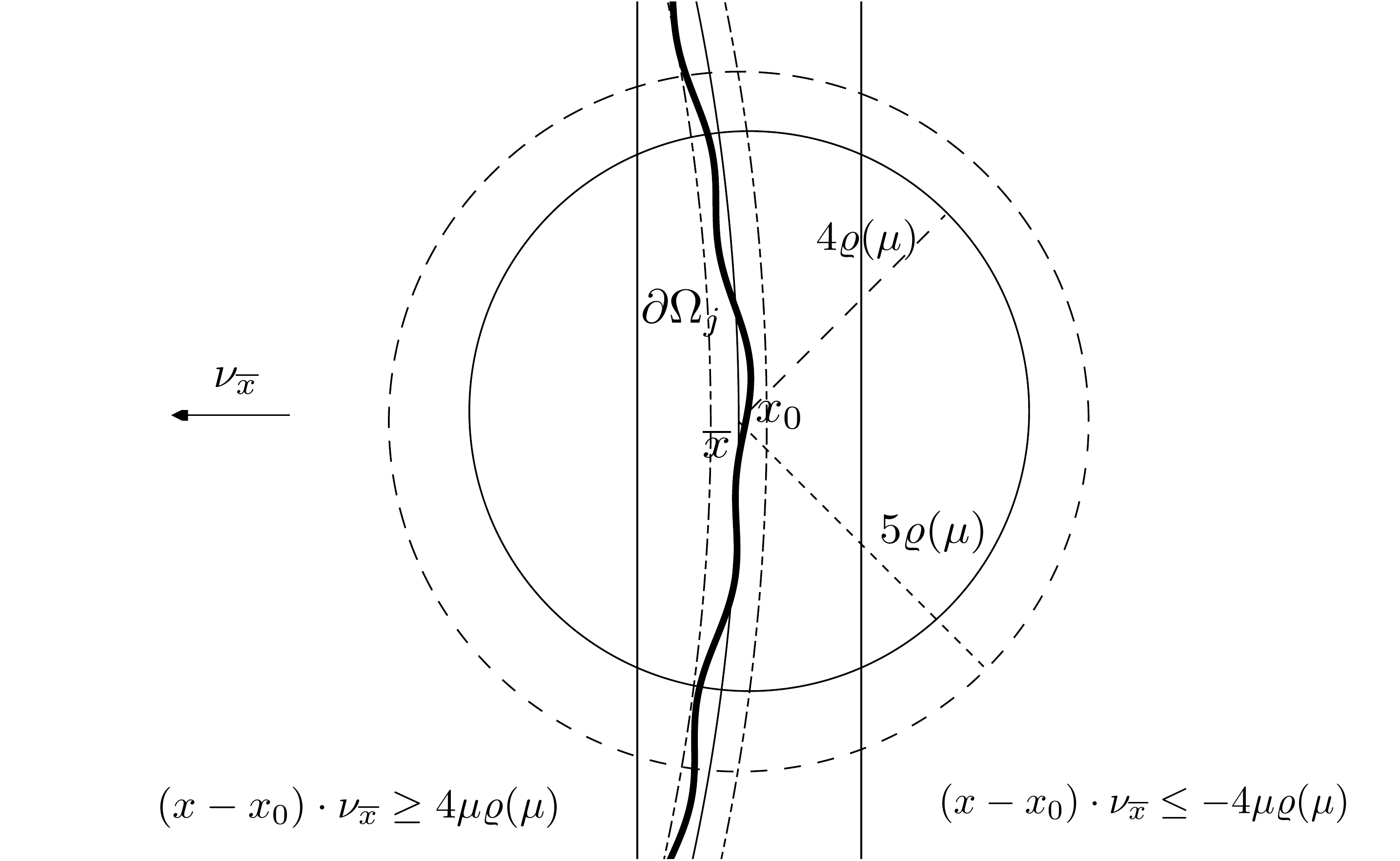}
\caption{The local construction for the proof of Proposition \ref{thm:sel}.}
\end{figure}
Since the balls \(\{B_{\varrho (\mu)}(\bar x)\}_{\overline x\in \partial B_1}\) cover \(\partial B_1\), it is not difficult to see that the above representations holds globally, i.e. for some functions \(g_j\) with uniformly bounded \(C^k\) norms
\begin{equation}\label{normalgraph}
\partial \Omega_j =\Big\{x+g_j(x)\,x\,:\, x\in \partial B_1\Big\}.
\end{equation}
Hence by the Ascoli-Arzel\`a Theorem and \eqref{kura} we obtain \(g_j\to 0\) in \(C^{k-1}(\partial B_1)\).
\par
We now define \(U_j=\lambda_j\,\Omega_j\) where \(\lambda^N_j =|B_1|/|\Omega_j|\). Clearly \(U_j\) still satisfies \(x_{U_j}=0\) and \(|U_j|=|B_1|\). Moreover, by Lemma \ref{lm:prop1} (i) we get $|\lambda_j-1|\le C\, \sigma^4\, \varepsilon_j$. Hence they are smoothly converging to \(B_1\). 
In order to verify \eqref{disfottuta}, we use Lemma \ref{lem:alphaprop} (ii) and Lemma \ref{lm:prop1} (i) to infer
\begin{equation*}
|\alpha(U_j)-\eps_j|\le |\alpha(\Omega_j)-\eps_j|+|\alpha(U_j)-\alpha(\Omega_j)|\le C(\sigma+\sigma^4)\,\eps_j\le C\,\sigma\, \eps_j.
\end{equation*}
Moreover, by the equation \eqref{f-prop} and Lemma \ref{lm:prop1} (iii)
\[
\F_{\widehat \eta}(U_j)-\F_{\widehat \eta} (B_1)\le (1+C\,\sigma^4\,\eps_j)\,\F_{\widehat \eta} (\Omega_j)-\F_{\widehat \eta}(B_1)\le C\,\sigma^4\,\eps_j,
\]
from which \eqref{disfottuta} immediately follows, since $|U_j|=|B_1|$ implies 
\[
\mathcal{F}_{\widehat\eta}(U_j)=E(U_j).
\]
This concludes the proof of the Selection Principle.
\end{proof}

\section{Final step: proof of the Main Theorem} 
\label{sec:second}
In this section we remove the restrictions in Theorem \ref{thm:stability_lim} and we give the proof of   the Main Theorem. For this we need two  preliminary results. The first one  is an $L^\infty$ estimate of the energy function outside a ball in terms of the measure of \(\Omega\) outside a smaller ball, based on a De Giorgi-type iteration technique.
The second one is a sub-optimal version of \eqref{goal} whose proof can be found in \cite{FMP1}.
\begin{lemma}
\label{lem:supest}
Let \(\Omega\) with be an open set with \(|\Omega|=|B_1|\) and let  \(u_\Omega\in W^{1,2}_0(\Omega)\) be its energy function. Then there exists a dimensional constant \(C_7\) such that for every \(R\ge 1\) we have
\begin{equation}\label{eq:supest}
\|u_\Omega\|_{L^\infty(\Omega\setminus B_{R+1})}\le C_7\, |\Omega \setminus B_R|^{1/N}.
\end{equation}
\end{lemma}

\begin{proof} Notice that \eqref{eq:supest} trivially holds if \(|\Omega\setminus B_R|=0\), hence we assume that \(|\Omega  \setminus B_R|>0\).
Let us define
\begin{equation*}
R_k=R+1-2^{-k},\qquad k\in\mathbb{N},
\end{equation*}
so that \(R_0=R\) and \(\lim_{k\to\infty} R_k=R+1\) and let us consider the following family of radial cut-off functions \(\varphi_k(x)=\phi_k(|x|)\), where \(\phi_k=0\) on \([0,R_{k-1}]\), \(\phi_k=1\) on \([R_k,+\infty)\) and it is  linear in between. 
Let us also define the following family of levels 
\begin{equation*}\label{sk}
s_k=M\,|\Omega \setminus B_R|^{1/N}\,(1-2^{-k}),
\end{equation*}
where \(M\) is a constant which  will be choosen later. Since \(u_\Omega\) satisfies
\begin{equation*}
\int \nabla u_\Omega\cdot \nabla v \,dx=\int v\, dx \qquad \mbox{ for every }\, v \in W^{1,2}_0(\Omega),
\end{equation*}
by inserting the test function \(v_k=\varphi^2_k\, (u_\Omega-s_k)_+ \in W^{1,2}_0(\Omega)\) standard computations lead 
\begin{equation}
\label{energy}
\int \big|\nabla(\varphi_k (u_\Omega-s_k)_+)\big|^2\, dx=\int \varphi^2_k\, (u_\Omega-s_k)_+\, dx +\int |\nabla \varphi_k|^2\,(u_\Omega-s_k)_+^2\, dx.
\end{equation}
By  \cite[Theorem 1]{Ta}, we have 
\begin{equation}\label{talenti}
\|u_\Omega\|_{L^\infty(\Omega)}\le \|u_{B_1}\|_{L^\infty(B_1)}\le C(N).
\end{equation}
Since \(|\nabla \varphi_k|\le 2^{k}\) and \(0\le \varphi_k\le 1\), by applying Sobolev inequality  \eqref{energy} and \eqref{talenti} we  infer
\begin{equation}
\label{enel}
\begin{split}
 \int \big(\varphi_k (u_\Omega-s_k)_+\big)^2\, dx&\le C\,\big|\big\{\varphi_k (u_\Omega-s_k)_+>0\big\}\big|^{2/N}\int \big|\nabla(\varphi_k (u_\Omega-s_k)_+)\big|^2\, dx\\
 &\le  C\big|\big\{\varphi_k (u_\Omega-s_k)_+>0\big\}\big|^{2/N}\,\left(\int \varphi^2_k (u_\Omega-s_k)_+\, dx\right.\\
&\left. +\int |\nabla \varphi_k|^2(u_\Omega-s_k)_+^2\, dx\right)\\
 &\le C\,  4^k \big|\big\{\varphi_k (u_\Omega-s_k)_+>0\big\}\big|^{1+2/N},
\end{split}
\end{equation}
where \(C\) depends only on \(N\). Since
\[
\big\{(u_\Omega-s_{k+1})_+>0\big\}\cap (\Omega\setminus B_{R_k})\subset \big\{(u_\Omega-s_{k})_+>s_{k+1}-s_k\big\}\cap(\Omega\setminus B_{R_k}),
\]
and \(s_{k+1}-s_k=M\, |\Omega\setminus B_R|^{1/N}\, 2^{-(k+1)}\), we obtain from \eqref{enel} 
\[
\begin{split}
\left|\big\{(u_\Omega-s_{k+1})_+>0\big\}\cap (\Omega\setminus B_{R_k})\right|&\le \frac{4^{k+1}\,(s_{k+1}-s_k)^2}{M^2 |\Omega\setminus B_R|^{2/N}}\,\big|\big\{(u_\Omega-s_k)_+>s_{k+1}-s_k\big\}\cap (\Omega\setminus B_{R_k})\big|\\
&\le \frac{4^{k+1}}{M^2 |\Omega\setminus B_R|^{2/N}}\,  \int \varphi_k^2 (u_\Omega-s_k)_+^2\, dx\\
& \le C \frac{16^{k}}{M^2 |\Omega\setminus B_R|^{2/N}}\,  \big|\big\{(u_\Omega-s_k)_+>0\big\}\cap (\Omega\setminus B_{R_{k-1}})\big|^{1+2/N}.
\end{split}
\]
By defining
\[
a_k=\frac{\big|\big\{(u_\Omega-s_k)_+>0\big\}\cap (\Omega\setminus B_{R_{k-1}})\big|}{|\Omega\setminus B_R|}\le 1,\qquad k\in\mathbb{N},
\]
we obtain the following  non-linear recursive equation for \(a_k\):
\[
a_{k+1}\le \frac{C}{M^2}\, 16^k\, a_{k}^{1+2/N}\qquad k\ge 1.
\]
Choosing \(M\) such that \(C/M^2=16^{-N}\) one easily sees by induction that
\[
a_k\le \left(\frac{1}{16^{N/2}}\right)^{k-1},\qquad k\in\mathbb{N},
\]
which clearly implies that $\lim_{k\to\infty} a_k=0$. By using the definition of $a_k$, this gives
\[
|\{(u-s_\infty)_+>0\}\cap (\Omega\setminus B_{R+1})|=0,
\]
with $s_\infty=M\, |\Omega\setminus B_R|^{1/N}=C^{1/2}\, 4^N\, |\Omega\setminus B_R|^{1/N}$. This gives the desired estimate \eqref{eq:supest}.
\end{proof}

\begin{lemma}\label{lm:nonsharp}
There exists a constant \(C_8=C_8(N)\) such that for every open set \(\Omega\subset\mathbb{R}^N\) with finite measure, there holds
\begin{equation}
\label{eq:nonsharp}
E(\Omega)\, |\Omega|^{-\frac{N+2}{N}}- E(B)\, |B|^{-\frac{N+2}{N}}\ge \frac{1}{C_8}\, \mathcal{A}(\Omega)^4.
\end{equation}
\end{lemma}
\begin{proof}
Recalling  that 
\[
E(\Omega)=-\frac{1}{2}\, \frac{1}{\lambda_{2,1}(\Omega)},
\]
where \(\lambda_{2,1}(\Omega)\) is defined in \eqref{autolavoro}, this result corresponds to taking $p=2$ and $q=1$ in \cite[Theorem 1]{FMP1}. 
\end{proof}
In what follows, we set 
\begin{equation}
\label{deficint}
D(\Omega)=E(\Omega)\, |\Omega|^{-\frac{N+2}{N}}- E(B)\, |B|^{-\frac{N+2}{N}},
\end{equation}
for notational simplicity.
\begin{lemma}
\label{bounded}
There exist constants \(C_9=C_9(N)\), \(\overline \delta=\overline \delta(N)>0\) and \(d=d(N)\) such that for every open set  \(\Omega\) with \(|\Omega|=|B_1|\) and \(D(\Omega)\le \overline  \delta(N)\), we can find another open set \(\widetilde \Omega\) with \(|\widetilde \Omega|=|B_1|\), \({\rm diam}(\widetilde \Omega)\le d\) and such that
\begin{equation}\label{eq:bounded}
\A(\Omega)\le \A(\widetilde \Omega)+C_9\,D(\Omega)\qquad\mbox{ and }\qquad D(\widetilde \Omega)\le C_9\, D(\Omega).
\end{equation}
\end{lemma}
\begin{proof}
Let us assume that the ball achieving the asymmetry is given by \(B_1\). By using this and the quantitative information \eqref{eq:nonsharp} we have, choosing \(\overline \delta (N)\) sufficiently small,
\begin{equation}\label{delta}
\frac{|\Omega\setminus B_1|}{|B_1|}\le \A(\Omega)\le C_8^{1/4}\,D(\Omega)^{1/4}\le C_8^{1/4}\,\overline \delta(N)^{1/4}<1.
\end{equation}
Let us now estimate the energy of \(\Omega\cap B_{k+2}\) for \(k\ge 0\). For every $k\in\mathbb{N}$, let \(\varphi_k\) be the cut-off function defined by
\[
\varphi_k(x)=\min\{1,(k+2-|x|)_+\},\qquad x\in\mathbb{R}^N,
\] 
which is supported in $B_{k+2}$ and is equal to $1$ in $B_{k+1}$. Then clearly 
\[
u_k=\varphi_k\, u_\Omega\in W^{1,2}_{0}(\Omega \cap B_{k+2}).
\] 
Hence, by using the equation satisfied by \(u\) and recalling that by \eqref{energyequivalent},
\[
E(\Omega)=-\frac 1 2\int u_\Omega\, dx, 
\]
we get
\begin{equation}
\label{torsionk}
\begin{split}
E(\Omega\cap B_{k+2})\le \frac 1 2 \int |\nabla u_k|^2 -\int u_k &=\frac 1 2\int \nabla u_\Omega\cdot \nabla (u_\Omega\, \varphi_k^2)\, dx+\frac{1}{2}\,\int |\nabla \varphi_k|^2 u^2 -\int \varphi_k u_\Omega\\
&=\frac 1 2 \int u_\Omega\,\varphi_k^2\, dx +\frac{1}{2}\,\int |\nabla \varphi_k|^2\, u_\Omega^2\, dx -\int \varphi_k\, u_\Omega\, dx\\
&= -\frac 1 2 \int u_\Omega+\frac{1}{2}\,\int (1-\varphi_k)^2\, u_\Omega\, dx+\frac{1}{2}\, \int |\nabla \varphi_k|^2 u_\Omega^2\, dx \\
&\le E(\Omega)+\frac{1}{2}\, |\Omega\setminus B_{k+1}\,|\left[\sup_{\Omega \setminus B_{k+1}} u_\Omega +\sup_{\Omega \setminus B_{k+1}} u_\Omega^2\right]\,.
\end{split}
\end{equation}
By setting 
\begin{equation}\label{bk}
b_k =\frac{|\Omega\setminus B_{k}|}{|B_1|},\qquad \mbox{ for every } k\ge 1,
\end{equation}
we have $b_k\le 1$ and of course $b_{k+1}\le b_k$. Hence  by recalling \eqref{eq:supest} we get 
\begin{equation}
\label{torsionk2}
\begin{split}
E(\Omega\cap B_{k+2})&\le E(\Omega)+C\, b_{k+1}\,\left[ b^{1/N}_{k}+b_{k}^{2/N}\right]\\
&\le E(\Omega)+C\, b_{k}^{1+\frac{1}{N}}.
\end{split}
\end{equation}
Using the definition of \(b_k\) and the Saint-Venant inequality \eqref{sv}, equation \eqref{torsionk2} implies
\[
\begin{split}
E(B_1)\big(1-b_{k+2})^{\frac{N+2}{N}}&=E(B_1)\,|B_1|^{-\frac{N+2}{N}}\,|\Omega \cap B_{k+2}|^{\frac{N+2}{N}}\\
&\le E(\Omega\cap B_{k+2})\le  E(\Omega)+C\, b_{k}^{1+\frac{1}{N}}\\
&= E(B_1)+D(\Omega)+C\, b^{1+1/N}_{k},
\end{split}
\]
where in the last estimate we used the very definition of deficit.
Hence, recalling that \(E(B_1)<0\), and assuming \(\overline \delta(N)\) sufficiently small we finally get
\begin{equation}
\label{rec}
b_{k+2}\le \frac{2N}{N+2}\, \Big(1-(1-b_{k+2})^{\frac{N+2}{N}}\Big)\le \widehat{C}\left(D(\Omega)+ b^{1+1/N}_{k}\right),
\end{equation}
for a suitable constant \(\widehat C\) depending only on the dimension \(N\). Let us now define
\begin{equation}
\label{defK}
\overline K=\max\{k\in \mathbb N: \text{ such that \(b_k\ge 2\,\widehat  C\, D(\Omega)\)}\},
\end{equation}
which exists since \(b_k\to 0\) as \(k\to \infty\). We claim that if we choose \(\overline\delta(N)\) sufficiently small then 
\begin{equation}
\label{kappasegnato}
\overline K\le K(N),
\end{equation}
for some  $K(N)$ depending  only on $N$. By noticing that for \(k+2\le \overline K\),  \eqref{rec} and \eqref{delta} give
\[
b_{k+2}\le 2\,( b_{k+2} -\widehat  C\,D(\Omega))\le 2\,\widehat  C\,b_{k}^{1+1/N}\le b_{k}^{1+1/2N},
\]
if $\overline \delta(N)\ll 1$.
Now, by iteration one easily notices that, as long as \(2\le k\le \overline K\),
\[
b_{k}\le b_1^{1+\frac{k}{2N}}. 
\]
Hence, by \eqref{delta} and \eqref{defK}, we deduce
\[
2\,\widehat  C\, D(\Omega)\le b_{\overline K} \le b_1^{1+\frac{\overline K}{2N}}\le \big(C\,D(\Omega)\big)^{\frac{1}{4}+\frac{\overline K}{8N}},
\]
which gives the desired estimate \eqref{kappasegnato}. 

By the definition of \(\overline K\), \eqref{defK}, and recalling the definition of \(b_k\),  we immediately see that
\begin{equation}\label{eq1}
|\Omega \cap B_{\overline K+3}|\ge |B_1|(1-b_{\overline K+1})\ge |B_1|-\widehat  C\,|B_1|\, D(\Omega),
\end{equation}
while \eqref{torsionk2} gives
\begin{equation}\label{eq2}
E(\Omega\cap B_{\overline K+3})\le E(\Omega)+C\, b_{\overline K+1}^{1+\frac{1}{N}}\le E(\Omega)+C\,D(\Omega)^{1+1/N}.
\end{equation}
Let us set
\[
 \widetilde \Omega=\frac{\Omega\cap B_{\overline K+3}}{r},\qquad \mbox{ where } \quad r=\left( \frac{\left|\Omega\cap B_{\overline K+3}\right|} {|B_1|}\right)^\frac{1}{N}.
\]
Clearly \(|\widetilde \Omega|=|B_1|\), moreover equation \eqref{eq1} implies
\begin{equation}
\label{raggetto}
1-\widehat C\, D(\Omega)\le r^N\le 1.
\end{equation}
Hence \eqref{kappasegnato}, \eqref{eq2} and \eqref{raggetto} give
\[
{\rm diam}(\widetilde \Omega)\le d(N)\qquad\text{and}\qquad D(\widetilde \Omega)\le CD(\Omega).
\]
In order to conclude the proof we only have to show that the estimate on the asymmetry in \eqref{eq:bounded} holds true. For this let \(B_1(x_0)\) be the optimal ball for \(\widetilde \Omega\) and let \(r\) be as above, so that \(|B_r(x_0)|=|\Omega \cap B_{\overline K+3}|\). By using \eqref{raggetto} and that $b_{\overline K+3}<2\, \widehat C\, D(\Omega)$ by definition of $\overline K$, we obtain 
\[
\begin{split}
|B_1|\,\A(\Omega)&\le |\Omega\Delta B_1(x_0/r)|\\
&\le |( \Omega\cap B_{\overline K+3})\Delta \Omega| + |( \Omega\cap B_{\overline K+3})\Delta B_r(x_0/r)|+|B_r(x_0/r)\Delta B_1(x_0/r)|\\
&\le C\,|B_1|\,D(\Omega)+r^N\, |\widetilde \Omega\Delta B_1(x_0)|+\omega_N\,(1-r^N)\\
&\le |B_1|\,\A(\widetilde \Omega)+C\,D(\Omega),
\end{split}
\]
which concludes the proof of the Lemma.
\end{proof}

We can finally prove the Main Theorem.
\begin{proof}[Proof of the Main Theorem] 
By Proposition \ref{prop:trick} it is enough to show that there exists a dimensional constant \(\sigma_E\) such that 
\[
E(\Omega)\, |\Omega|^{-\frac{N+2}{N}}- E(B)\, |B|^{-\frac{N+2}{N}}\ge \sigma_E\,\mathcal{A}(\Omega)^2\qquad \text{for all sets \(\Omega\).}
\]
Also, since the above inequality is scaling invariant, we can assume that $|\Omega|=|B_1|$, without loss of generality.
For notational simplicity, we keep on using the notation $D$ introduced in \eqref{deficint}.
\par
Let \(\overline \delta(N)\le 1\) be the constant appearing in Lemma \ref{bounded}. If $D(\Omega)\ge \overline  \delta(N)$, then since $\mathcal{A}(\Omega)<2$ we get
\[
D(\Omega)\ge \frac{\overline \delta(N)}{4}\, \mathcal{A}(\Omega)^2.
\]
Thus we can suppose that $D(\Omega)<\overline \delta(N)$. Thanks to Lemma \ref{bounded}, we can construct a new open set \(\widetilde \Omega\) with \({\rm diam }(\widetilde \Omega)\le d(N)\) and $|\widetilde\Omega|=|B_1|$, which satisfies \eqref{eq:bounded}.  Up to a translation we have $\widetilde\Omega\subset B_{d(N)}$, then by 
applying Theorem \ref{thm:stability_lim} with $R=d(N)$ we have
\[
D(\widetilde \Omega)\ge \widehat \sigma(d(N))\,\alpha(\widetilde \Omega),\qquad \mbox{ if }\quad\alpha(\widetilde\Omega)\le \widehat \varepsilon(d(N)).
\]
By appealing to Lemma \ref{lem:alphaprop} (i) and to the very definition of Fraenkel asymmetry, the previous implies
\begin{equation}
\label{piccolo}
D(\widetilde \Omega)\ge \frac{\widehat \sigma(d(N))\, |B_1|^2}{C_1}\,\mathcal{A}(\widetilde \Omega)^2,\qquad \mbox{ if }\quad\alpha(\widetilde\Omega)\le \widehat \varepsilon(d(N)).
\end{equation}
On the other hand, in the case \(\alpha(\widetilde \Omega)>\widehat \varepsilon(d(N))\), let $B$ be the ball (of radius $1$) such that $\mathcal{A}(\widetilde\Omega)=|\widetilde\Omega\Delta B|/|B|$. Since \(\widetilde \Omega\subset B_{d(N)}\) it is immediate to check that \(B\subset B_{d(N)+3}\), hence Lemma \eqref{lem:alphaprop} (ii) and \(\alpha(B)=0\)  give 
\[
\frac{\widehat \varepsilon(d(N))}{C_2(d(N)+3)\, |B_1|}<\frac{\alpha(\widetilde \Omega)}{C_2(d(N)+3)\, |B_1|}\le \, \frac{|\widetilde\Omega\Delta B|}{|B_1|}=\mathcal{A}(\widetilde\Omega).
\]
Estimate \eqref{eq:nonsharp} now implies\footnote{We note that in this part of the argument it is not really need the power law relation between \(D(\Omega)\) and \(\mathcal A(\Omega)\) given by \eqref{eq:nonsharp}, it would be sufficient to know that \(\A(\Omega)\to 0\) as \(D(\Omega)\to 0\).}
\begin{equation}
\label{grande}
D(\widetilde \Omega)\ge \frac{\A(\widetilde\Omega)^4}{C_8}\ge\left(\frac{\widehat \varepsilon(d(N))}{C_2(d(N)+3)\, |B_1|}\right)^2\, \frac{\A(\widetilde \Omega)^2}{C_8}, \qquad \mbox{ if }\quad\alpha(\widetilde\Omega)>\widehat \varepsilon(d(N)).
\end{equation}
Setting 
\[
c=\min\left\{\frac{\widehat \sigma(d(N))\, |B_1|^2}{C_1},\,\frac{1}{C_8}\left(\frac{\widehat \varepsilon(d(N))}{C_2(d(N)+3)\, |B_1|}\right)^2\right\},
\] 
thanks to \eqref{eq:bounded}, \eqref{piccolo} and \eqref{grande} and since  \(\overline \delta(N)\le 1\) we get
\[
\begin{split}
\frac {c}{2} \mathcal A(\Omega)&\le \frac {c}{2}\, \Big(\mathcal A(\widetilde \Omega)+C_9\,D(\Omega)\Big)^2\le c\,\mathcal A (\widetilde \Omega)^2+c\, C_9^2 \,D(\Omega)^2\\
&\le D(\widetilde \Omega)+c\, C_9^2\, D(\Omega)^2\le C_9\, D(\Omega)+c\, C_9^2\, D(\Omega)^2\le C_9 \,(1+c\,C_9)\,D(\Omega).
\end{split}
\]
If we now define 
\[
\sigma_E=\min\left\{\frac{c}{2\, C_9\, (1+c\, C_9)},\frac{\overline\delta(N)}{4}\right\},
\]
we get the desired conclusion.
\end{proof}

\appendix

\section{Proof of Lemma \ref{dambrin}}
\label{sec:shape}
In this Appendix we briefly sketch the proof of Lemma \ref{dambrin}, referring to \cite{Da} for more details.  We start with the following:
\begin{lemma}\label{vectorfield}Given \(\gamma\in (0,1]\) there exists \(\delta_4=\delta_4(N,\gamma)>0\) and a modulus of continuity \(\widehat\omega\) such that   for  every nearly spherical set \(\Omega\)  parametrized by \(\varphi\) with \(\|\varphi\|_{C^{2,\gamma}(\partial B_1)}\le \delta_4\) and \(|\Omega|=|B_1|\), we can find an autonomous vector field \(X_{\varphi}\) for which the following holds true:
\begin{enumerate}
\item[(i)] \({\rm div} X_{\varphi} =0\) in a \(\delta_4\)-neighborhood of \(\partial B_1\);
\item [(ii)]if \(\Phi_t:=\Phi(t,x)\) is the flow of \(X_{\varphi}\), i.e.
\[
\partial_t \Phi_t=X_{\varphi}(\Phi_t) \qquad \Phi_0(x)=x,
\] 
then \(\Phi_1(\partial B_1)=\partial \Omega\) and \(|\Phi_t(B_1)|=|B_1|\) for all \(t\in [0,1]\). 

\item[(iii)] We have
\begin{equation}\label{c2close}
\big\|\Phi_t-{\rm Id}\big\|_{C^{2,\gamma}}\le  \widehat \omega \big(\|\varphi\|_{C^{2,\gamma}(\partial B_1)}\big)\qquad \mbox{ for every } \, t\in [0,1],
\end{equation}
\begin{equation}
\label{hunmezzovicine}
\big\|\varphi-(X_\varphi \cdot \nu_{ B_1})\big\|_{H^{1/2}(\partial B_1)}\le \widehat  \omega \big(\|\varphi\|_{L^\infty(\partial B_1)}\big) \|\varphi\|_{H^{1/2}(\partial B_1)}.
\end{equation}
and 
\begin{equation}
\label{vero}
(X\cdot\theta)\circ\Phi_t-X \cdot \nu_{B_1}=(X \cdot \nu_{B_1})\,\psi_t\qquad \text{on \( \partial B_1\)}
\end{equation}
with \(\|\psi_t\|_{C^{2,\gamma}(\partial B_1)}\le \widehat \omega (\|\varphi\|_{C^{2,\gamma}(\partial B_1)})\).
\end{enumerate}
\end{lemma}

\begin{proof}The construction is general and can be done in the neighborhood of every sufficiently smooth set, see \cite[Proposition 1]{Da} and \cite[Theorem 3.7]{AFM}. In the case of the ball we can however give an explicit expression for \(X_{\varphi}\) and for its flow \(\Phi_t\) in a neighborhoodd of \(\partial B_1\).  More precisely in polar coordinates, \(x=\varrho\,\theta\), \(\varrho=|x|\), \(\theta=x/|x| \in \mathbb \partial B_1\), we define for \(|\varrho-1|\le \delta_4\ll1\),
\begin{equation}\label{explicit}
X_{\varphi}(\rho,\theta)=\frac{\big(1+\varphi(\theta)\big)^N-1}{N\rho^{N-1}}\,\theta \qquad \Phi_t(\rho,\theta)= \Big[\rho^{N}+t\Big(\big(1+\varphi(\theta)\big)^N-1\Big)\Big]^{1/N}\theta,
\end{equation}
and we globally extend the  vectorfield (and hence the flow) in order to satisfy \eqref{c2close}.
In this way points (i) (ii) and equations \eqref{c2close} and \eqref{vero} follow by direct computation. For equation \eqref{hunmezzovicine} notice that, on \(\partial B_1\)
\[
\varphi-X_\varphi \cdot \nu_{\partial B_1}=\frac{1}{N}\,\sum_{h=2}^N \binom{N}{h} \varphi^h.
\]
Since  harmonic functions minimize the Dirichlet energy with respect to their own boundary data we get (recalling the notations of Definition \ref{def:hunmezzo})
\[
\int_{B_1} \left|\nabla H\big( \varphi^h\big)\right|^2\, dx\le \int_{B_1} \left|\nabla\big(H(\varphi)\big)^h\right|^2\, dx\qquad \mbox{ for every } h\ge 1.
\]
Hence, a straightforward computation gives
\[
\big\|\varphi-(X_\varphi \cdot \nu_{\partial B_1})\big\|_{H^{1/2}(\partial B_1)}\le C(N)\Big(\|H(\varphi)\|_{L^\infty(B_1)}+\|\varphi\|_{L^\infty(\partial B_1)}\Big)\|\varphi\|_{H^{1/2}(\partial B_1)}.
\]
Since, by the maximum principle, \(\|H(\varphi)\|_{L^\infty(B_1)}\le \|\varphi\|_{L^\infty(\partial B_1)}\), we conclude the proof.
\end{proof}

With \(\Phi_t\) and \(X_{\varphi}\) as above, we now  set \(\Omega_t=\Phi_t(B_1)\) and 
\[
e(t):=E(\Omega_t)=-\frac{1}{2}\int_{\Omega_t} |\nabla u_t|^2\,dx,
\]
where \(u_t=u_{\Omega_t}\) is the energy function of \(\Omega_t\), i.e.
\begin{equation}\label{energyt}
\begin{cases}
-\Delta u_t=1\quad&\text{in \(\Omega_t\)}\\
u_t=0&\text{on \(\partial \Omega_t.\)}
\end{cases}
\end{equation} 
We want to compute \(e'(t)\) and \(e''(t)\). For this we recall that the map \(t\mapsto u_t\) is differentiable, see for instance \cite[Theorem 5.3.1]{HP}, and that its derivative \(\dot{u}_t\) satisfies
\begin{equation}\label{energyder}
\begin{cases}
-\Delta \dot u_t=0\quad&\text{in \(\Omega_t\)}\\
\dot u_t=-\nabla u_t \cdot X_{\varphi}&\text{on \(\partial \Omega_t\)}.
\end{cases}
\end{equation}
Recalling Hadamard formula (see \cite[Section 5.2]{HP}), 
for every $f$ sufficiently smooth
\begin{equation*}
 \frac{d}{dt} \int_{\Omega_t} f(t,x) \, dx=\int_{\Omega_t} \partial_t f(t,x) \,dx+\int_{\partial \Omega_t} f(t,x) (X_{\varphi}\cdot \nu_{\partial\Omega_t})\, d\mathcal H^{N-1},
\end{equation*}
we can now compute (dropping the subscript \(\varphi\) for notational simplicity)
\begin{equation}\label{derivata1}
\begin{split}
e'(t)&=-\int_{\Omega_t}\nabla u_t\cdot \nabla \dot u_t\, dx-\frac 1 2 \int_{\partial \Omega_t} |\nabla u _t|^2 (X\cdot \nu_{\Omega_t}) \, d\mathcal H^{N-1}\\
&=-\frac 1 2 \int_{\partial \Omega_t} |\nabla u _t|^2 (X\cdot \nu_{\Omega_t}) \, d\mathcal H^{N-1}=-\frac{1}{2} \int_{ \Omega_t}{\rm div}\Big( |\nabla u _t|^2 X\Big) \, d x,
\end{split}
\end{equation}
where we have used that since  \(\dot u_t\) is harmonic and \(u_t\in W^{1.2}_0(\Omega_t)\), their gradient are \(L^2\) orthogonal. Differentiating again, using Hadamard formula and that  \(X\) is autonomous we get
\begin{equation}
\label{der2.1}
\begin{split}
e''(t)=&-  \int_{\partial \Omega_t} \big( \nabla u_t\cdot \nabla \dot u_t \big)\big(X \cdot \nu_{ \Omega_t}\big) d\mathcal H^{N-1}-\frac 1 2 \int_{\partial \Omega_t} {\rm div}\Big( |\nabla u _t|^2 X\Big)\big(X \cdot \nu_{ \Omega_t}\big) \, d\mathcal H^{N-1}.
\end{split}
\end{equation}
Since \(u_t\) is positive and vanishes on \(\partial \Omega_t\),
\begin{equation}
\label{chepalle}
\nabla u_t =-|\nabla u_t|\nu_{\Omega_t}\qquad \text{on \(\partial \Omega_t\)},
\end{equation} 
hence
\[
\nabla u_t\cdot \nabla \dot u_t=-|\nabla u_t| \cdot(\nabla \dot u_t\cdot \nu_{\Omega_t}\big).
\]
Moreover, by \eqref{energyder} and \eqref{chepalle}, \(\dot u_t=|\nabla u_t|\big(X \cdot \nu_{ \Omega_t}\big)\), so that equation \eqref{der2.1} becomes
\begin{equation}\label{der2.2}
\begin{split}
e''(t)&=  \int_{\partial \Omega_t}\dot u_t\, \partial_{\nu} \dot u_t\,d\mathcal H^{N-1}- \int_{\partial \Omega_t}\big(X \cdot \nu_{ \Omega_t}\big)\big(\nabla^2 u_t[\nabla u_t]\cdot X\big)\,d\mathcal H^{N-1}
-\frac 1 2 \int_{\partial \Omega_t} |\nabla u _t|^2{\rm div} (X)\big(X\cdot \nu_{ \Omega_t}\big) \\
&= \int_{\partial \Omega_t}\dot u_t\, \partial_{\nu} \dot u_t\,d\mathcal H^{N-1}- \int_{\partial \Omega_t}\big(X \cdot \nu_{ \Omega_t}\big)\big(\nabla^2 u_t[\nabla u_t]\cdot X\big)\,d\mathcal H^{N-1}
\end{split}
\end{equation}
where \(\partial_{\nu} \dot u_t=\nabla \dot u_t \cdot \nu_{\Omega_t}\) is the normal derivative of \(\dot u_t\) and we have used that \({\rm div} X=0\) in a neighborhood of \(\partial B_1\) (where \(\partial \Omega_t\) is contained). Since on \(\partial \Omega_t=\{u_t=0\}\) we have
\[
-1=\Delta u_t=-|\nabla u|\mathscr H_{\partial \Omega_t}+\nabla^2 u_t[\nu_{\partial\Omega_t}]\cdot \nu_{\partial\Omega_t}
\]
where \(\mathscr H_{\partial \Omega_t}\) is the mean curvature of \(\partial \Omega_t\) computed with respect to the exterior normal. Hence, we finally get, taking also into account \eqref{chepalle} and defining \(X^\tau=X-(X\cdot \nu)\,\nu\) the tangential component of \(X\),
\begin{equation}
\label{der2fine}
\begin{split}
e''(t)&=\int_{\partial \Omega_t}\dot u_t\, \partial_{\nu} \dot u_t\,d\mathcal H^{N-1}+ \int_{\partial \Omega_t}|\nabla u_t|\big(X \cdot \nu_{ \Omega_t}\big)^2\big(\nabla^2 u_t[\nu_{\partial\Omega_t}]\cdot \nu_{\partial\Omega_t}\big)\,d\mathcal H^{N-1}\\
&- \int_{\partial \Omega_t}\big(X \cdot \nu_{ \Omega_t}\big)\big(\nabla^2 u_t[\nabla u_t]\cdot X^\tau\big)\,d\mathcal H^{N-1}\\
&=\int_{ \Omega_t}|\nabla \dot u_t|^2\, dx - \int_{\partial \Omega_t}\big(X \cdot \nu_{ \Omega_t}\big)^2|\nabla u_t|\,d\mathcal H^{N-1}+\int_{\partial \Omega_t}\big(X \cdot \nu_{ \Omega_t}\big)^2|\nabla u_t|^2\mathscr  H_{\partial \Omega_t}\,d\mathcal H^{N-1}\\
&- \int_{\partial \Omega_t}\big(X \cdot \nu_{ \Omega_t}\big)\big(\nabla^2 u_t[\nabla u_t]\cdot X^\tau\big)\,d\mathcal H^{N-1}.
\end{split}
\end{equation}
Notice that  in the last equality we have used Green formula in the first term (recall the \(\dot u_t\) is harmonic). We now observe that \(\mathscr  H_{\partial B_1}=(N-1)\), \(X^{\tau}=0\) on \(\partial B_1\),
\[
u_0=u_{B_1}=\frac{1-|x|^2}{2N} \qquad\text{in \(B_1\)},
\]
and that\footnote{Here we are using the notations of Definition \ref{def:hunmezzo}}  \(\dot u_0=H(X\cdot \nabla u_0)\). By using these facts in equation \eqref{der2fine} at $t=0$, we get
\begin{equation}
\label{palladerivata}
\begin{split}
e''(0)&=\int_{ B_1} |\nabla H(X\cdot\nabla u_0)|^2\,dx
+\int_{\partial B_1} \Big[(N-1)\, |\nabla u_0|-1\Big]\, |\nabla u_0|\,\big(X\cdot \nu_{ B_1})^2\,d\mathcal H^{N-1}\\
&=\frac {1}{N^2} \Bigg(\int_{ B_1} |\nabla H(X\cdot\nu_{ B_1})|^2\,dx-\int_{ B_1}\big(X\cdot\nu_{B_1})^2\,d\mathcal H^{N-1}\Bigg).
\end{split}
\end{equation}
\begin{lemma}
\label{dambrin2}
Let \(\gamma\in (0,1]\), there exist \(\delta_5=\delta_5(N,\gamma)\) and a modulus of continuity \(\widetilde \omega\) such that if \(\Omega\), \(\varphi\), \(X_\varphi\) and \(\Phi_t\) are as in Lemma \ref{vectorfield} and \(\|\varphi\|_{C^{2,\gamma}}\le \delta_5\), then
\begin{equation}
\label{evai}
|e''(t)-e''(0)|\le\widetilde \omega\big(\|\varphi\|_{C^{2,\gamma}}\big)\big\|X_\varphi\cdot \nu_{ B_1}\|^2_{H^{1/2}(\partial B_1)}.
\end{equation}
\end{lemma}

\begin{proof}We start from \eqref{der2fine} and pull it back on \( B_1\) through \(\Phi_t\):
\begin{equation}
\begin{split}
e''(t)&=\int_{B_1}\big|\nabla \dot u_t\big|^2\circ \Phi_t \,\det \nabla \Phi_t\, dx\\
&- \int_{\partial  B_1}\Big\{\big(X \cdot \nu_{ \Omega_t}\big)^2|\nabla u_t|-\big(X \cdot \nu_{ \Omega_t}\big)^2|\nabla u_t|^2\,\mathscr{H}_{\partial \Omega_t}\Big\} \circ \Phi_t \, J^{\partial B_1}\Phi_t\,d\mathcal H^{N-1}\\
&- \int_{\partial B_1}\Big\{\big(X \cdot \nu_{ \Omega_t}\big)\big(\nabla^2 u_t[\nabla u_t]\cdot X^\tau\big)\Big\}\circ \Phi_t\, J^{\partial B_1}\Phi_t \,d\mathcal H^{N-1}\\
&:=I_1(t)+I_2(t)+I_3(t),
\end{split}
\end{equation}
where \(J^{\partial B_1}\Phi_t\) is the \emph{tangential Jacobian} of \(\Phi_t\) (see \cite[Section 11.1]{Maggi}) and we have dropped the subscript \(\varphi\) for notational simplicity.  By \eqref{c2close} we get   
\begin{equation}
\label{facili}
\big \|\mathscr  H_{\partial \Omega_t}\circ \Phi_t-\mathscr  H_{\partial B_1}\big\|_{L^\infty(\partial B_1)}+\big\|J^{\partial B_1}\Phi_t-1\big\|_{L^\infty(\partial B_1)}+\big\|\det \nabla \Phi_t-1\big\|_{L^\infty( B_1)}\le \omega\big(\|\varphi\|_{C^{2,\gamma}}\big).
\end{equation}
where here and in the following \(\omega\) will just denote a modulus of continuity whose precise expression will change line by line.
Moreover pulling back to \(B_1\) the equation satisfied by \(u_t\), i.e. considering the equation satisfyied by \(u_t\circ \Phi_t\) on \(B_1\), Schauder estimates give
\begin{equation}
\label{menofacili}
\|u-u_t\circ \Phi_t\|_{C^{2,\gamma}(\overline{B_1})}\le  \omega\big(\|\varphi\|_{C^{2,\gamma}}\big).
\end{equation}
By Lemma \ref{vectorfield} (i),  \(X\) is parallel to \(\theta=x/|x|\) in neighborhood of \(\partial B_1\), hence
\begin{equation}
\label{norm}
\begin{split}
\big|(X\cdot \nu_{\Omega_t})\circ \Phi_t-X \cdot \nu_{B_1}\big|&=\big|\big((X\cdot\theta)\circ\Phi_t\big)\,\big((\theta \cdot \nu_{\Omega_t})\circ \Phi_t\big)-X \cdot \nu_{B_1}\big|\\
&\le \big|(X\cdot\theta)\circ\Phi_t\big|\,\big|(\theta \cdot \nu_{\Omega_t})\circ \Phi_t-1\big|+\big|(X\cdot\theta)\circ\Phi_t-X \cdot \nu_{B_1}\big|\\
& \le  \omega\big(\|\varphi\|_{C^{2,\gamma}}\big)\, |X\cdot \nu_{B_1}|.
\end{split}
\end{equation}
where in the last inequality we have used \eqref{vero}.
 With the same computations we also get,
\begin{equation}\label{tan}
|X^\tau\circ \Phi_t|\le \omega\big(\|\varphi\|_{C^{2,\gamma}}\big)\, |X\cdot \nu_{B_1}|.
\end{equation}
By \eqref{facili}--\eqref{tan} we deduce
\begin{equation}
\label{tantafaticaperniente}
\big |I_2(t)-I_2(0)\big |+\big | I_3(t)\big|\le {\omega}\big(\|\varphi\|_{C^{2,\gamma}}\big)\, \big\|X\cdot \nu_{B_1}\big\|^2_{L^2(\partial B_1)}.
\end{equation}
Since $I_3(0)=0$, we are left to estimate \(I_1(t)-I_1(0)\). Defining \(v_t=\dot u_t\circ \Phi_t\) we have \(\nabla v_t=(\nabla \Phi_t)^T\,\nabla u_t\circ \Phi_t\), where $M^T$ denotes the transposition of a matrix $M$.
Hence, taking into account \eqref{c2close} and \eqref{facili}, it is an easy computation to  see that the proof of the Lemma will be concluded once we have shown that
\begin{equation}
\label{moccolo}
\Big|\int_{B_1} |\nabla v_t|^2-|\nabla \dot u_0|^2 \, dx\Big|\le   \ {\omega}\big(\|\varphi\|_{C^{2,\gamma}}\big) \big\|X\cdot \nu_{B_1}\big\|^{2}_{H^{1/2}(\partial B_1)}.
\end{equation}
Now, from \eqref{energyder}, we see that  \(v_t\) solves the linear elliptic problem
\[
\begin{cases}
\mathrm{div} \Big( M_t\,\nabla v_t\Big)=0\quad &\text{in \(B_1\)},\\
v_t=-(\nabla u_t\cdot X)\circ \Phi_t &\text{on \(\partial B_1\)},
\end{cases}
\]
where $M_t$ is the symmetric positive definite matrix given by
\[
M_t=\det \nabla \Phi_t\, \big((\nabla \Phi_t)^{-1}\big)^T\,(\nabla \Phi_t)^{-1} .
\]
Hence, classical elliptic estimates  together with \eqref{c2close} give
\begin{equation}
\label{moccolo2}
\begin{split}
\big\|\nabla v_t&-\nabla \dot u_0\big\|_{L^2(B_1)}\\
&\le C(N)\,\Big( \big\|\big(M_t -{\rm Id}\big)\,\nabla v_t\big\|_{L^2(B_1)}+ \big\|(\nabla u_t\cdot X)\circ \Phi_t -\nabla u_0\cdot X \big\|_{H^{1/2}(\partial B_1)}\Big)\\
&\le \omega\big(\|\varphi\|_{C^{2,\gamma}}\big)\, \big\|\nabla v_t\big\|_{L^{2}( B_1)}+ C(N)\, \big\|(\nabla u_t\cdot X)\circ \Phi_t -\nabla u_0\cdot X \big\|_{H^{1/2}(\partial B_1)}.
\end{split}
\end{equation}
Now by Lemma \ref{vectorfield} (i) \(X=(X\cdot \theta)\,\theta\), where \(\theta=x/|x|\).  Since  \(\nabla u_0=-|\nabla u_0|\,\theta\) on \(\partial B_1\), by using \eqref{vero} and \eqref{menofacili} we get
\begin{equation}
\label{moccolo3}
\begin{split}
\big\|(\nabla u_t\cdot X)\circ \Phi_t& -\nabla u_0\cdot X \big\|_{H^{1/2}(\partial B_1)}\\
&\le \big\|\big\{( \nabla u_t \cdot \theta)\circ \Phi_t -(\nabla u_0\cdot\theta) \big\}(X\cdot \theta)\circ \Phi_t \big\|_{H^{1/2}(\partial B_1)}\\
&\quad+\big\||\nabla u_0|\, \big((X\cdot \theta)\circ \Phi_t -X\cdot \nu_{B_1}\big) \big\|_{H^{1/2}(\partial B_1)}\\
&\le \big\|\big\{\big(\nabla u_t\circ\Phi_t-\nabla (u_t\circ\Phi_t)\big) \cdot( \theta\circ \Phi_t) \big\}(X\cdot \theta)\circ \Phi_t \big\|_{H^{1/2}(\partial B_1)}\\
&\quad +\big\|\big\{\nabla (u_t\circ\Phi_t)\cdot(\theta\circ\Phi_t) -(\nabla u_0\cdot\theta) \big\}(X\cdot \theta)\circ \Phi_t \big\|_{H^{1/2}(\partial B_1)}\\
&\quad +\big\| |\nabla u_0|\,\big((X\cdot \theta)\circ \Phi_t -X\cdot \nu_{B_1}\big) \big\|_{H^{1/2}(\partial B_1)}\\
&\le \omega\big(\|\varphi\|_{C^{2,\gamma}}\big) \big\|\nabla u_0\cdot X \big\|_{H^{1/2}(\partial B_1)}.
\end{split}
\end{equation}
Equations \eqref{moccolo2} and \eqref{moccolo3} imply
\begin{equation}
\begin{split}
\big\|\nabla v_t&-\nabla \dot u_0\big\|_{L^2(B_1)}\\
& \le \omega\big(\|\varphi\|_{C^{2,\gamma}}\big)\, \left(\big\|\nabla v_t\big\|_{L^{2}( B_1)}
+\big\|\nabla u_0\cdot X \big\|_{H^{1/2}(\partial B_1)}\right)\\
&\le\omega\big(\|\varphi\|_{C^{2,\gamma}}\big)\,\left( \big\|\nabla v_t-\nabla \dot u_0\big\|_{L^2(B_1)}
+\big\|\nabla \dot u_0 \big\|_{L^{2}( B_1)}
+\big\|\nabla u_0\cdot X \big\|_{H^{1/2}(\partial B_1)}\right)\\
&\le  \omega \big(\|\varphi\|_{C^{2,\gamma}}\big)\, \left(\big\|\nabla v_t-\nabla \dot u_0\big\|_{L^2(B_1)}+2\, \big\|\nabla u_0\cdot X \big\|_{H^{1/2}(\partial B_1)}\right),
\end{split}
\end{equation}
where in the last inequality we have used Definition \ref{def:hunmezzo}, since \(\dot u_0=-H(\nabla u_0\cdot X)\).
Choosing \(\delta_5\)  so that \(\omega\big(\|\varphi\|_{C^{2,\gamma}}\big)\le1/2\),  we finally get
\begin{equation}\label{nonfiniscepiu}
\left\|\nabla v_t-\nabla \dot u_0\right\|_{L^2(B_1)}\le 4\, \omega\big(\|\varphi\|_{C^{2,\gamma}}\big) \big\|\nabla u_0\cdot X \big\|_{H^{1/2}(\partial B_1)}.
\end{equation}
Since, clearly
\[
\begin{split}
\Big|\int_{B_1} |\nabla v_t|^2-|\nabla \dot u_0|^2 \, dx\Big|&\le\|\nabla v_t -\nabla \dot u_0\|_{L^2(B_1)}\|\nabla v_t+\nabla\dot u_0\|_{L^2(B_1)}\\
&\le 2\,\|\nabla \dot u_0\|_{L^2(B_1)}\,\|\nabla v_t -\nabla \dot u_0\|_{L^2(B_1)}+\|\nabla v_t -\nabla \dot u_0\|_{L^2(B_1)}^2,
\end{split}
\]
equation \eqref{moccolo} follows from \eqref{nonfiniscepiu} and our definition of \(H^{1/2}\) norm.
\end{proof}
We can now prove Lemma \ref{dambrin}.
\begin{proof}[Proof of Lemma \ref{dambrin}.]
By Taylor formula,
\begin{equation}
\label{taylorint}
E(\Omega)=E(B_1)+e'(0)+\frac 1 2 e''(0)+\int_0^1 (1-s) \big( e''(s)-e(0)\big)\, ds.
\end{equation}
Since \(|\Omega_t|=|B_1|\), by the Saint-Venant inequality we have \(e'(0)=0\). Equation \eqref{palladerivata} gives
\begin{equation*}
\begin{split}
e''(0)&=\frac {1}{N^2} \Big(\int_{ B_1} |\nabla H(X\cdot\nu_{ B_1})|^2\,dx-\int_{ B_1}\big(X\cdot\nu_{B_1})^2\,d\mathcal H^{N-1}\Big)\\
&=\partial^2 E(B_1)\big[X\cdot\nu_{ B_1},X\cdot\nu_{ B_1}\big].
\end{split}
\end{equation*}
Since,
\begin{equation*}
\begin{split}
\Big|\partial^2 E(B_1)\big[X\cdot\nu_{ B_1},X\cdot\nu_{ B_1}\big]-\partial^2 E(B_1)\big[\varphi,\varphi\big]\Big|&\le \|X\cdot \nu_{B_1}-\varphi\|_{H^{1/2}(\partial B_1)}\|X\cdot \nu_{ B_1}+\varphi\|_{H^{1/2}(\partial B_1)}
\end{split}
\end{equation*}
equation \eqref{taylor}  follows by  \eqref{hunmezzovicine}, \eqref{taylorint} and \eqref{evai}.
\end{proof}

\begin{ack}
It is a pleasure to acknowledge Marco Barchiesi, Nicola Fusco and Aldo Pratelli for some useful discussions at a preliminary stage of this work. We also thank Jimmy Lamboley for pointing out to us paper \cite{Da}. Nikolai Nadirashvili kindly provided us a copy of \cite{HN} and \cite{Na}, we wish to warmly thank him. A visit of the second author to Marseille has been the occasion to fix some final details: the LATP institution and its facilities are kindly acknowledged. 
\end{ack}

\end{document}